\documentclass[10pt]{amsart}
\usepackage{latexsym, amsmath,amssymb}

\renewcommand{\epsilon}{\varepsilon}

\newtheorem{theorem}{Theorem}

\newtheorem{lemma}[theorem]{Lemma}
\newtheorem{corr}[theorem]{Corollary}

\newtheorem{proposition}[theorem]{Proposition}
\newtheorem{deff}[theorem]{Definition}

\newcommand{\bth}{\begin{theorem}}
	\newcommand{\ble}{\begin{lemma}}
		\newcommand{\bcor}{\begin{corr}}

			\newcommand{\bdeff}{\begin{deff}}

				\newcommand{\bprop}{\begin{proposition}}
					\newcommand{\ele}{\end{lemma}}
				\newcommand{\ecor}{\end{corr}}
			\newcommand{\edeff}{\end{deff}}
		
		\newcommand{\eprop}{\end{proposition}}

	\newcommand{\Rn}{{\mathbb R}^n}
	
	\newcommand{\la}{\lambda}
	
	\newcommand{\eps}{\varepsilon}
	\newcommand{\e}{\varepsilon}

	\renewcommand{\Pi}{\varPi}

	\renewcommand{\epsilon}{\varepsilon}

	\newcommand{\R}{{\mathbb R}}
	\newcommand{\ls}{\lesssim}
	\newcommand{\gs}{\gtrsim}
	\newcommand{\1}{{\rm 1\hspace*{-0.4ex}%
			\rule{0.1ex}{1.52ex}\hspace*{0.2ex}}}
	
	\newcommand{\ola}{\1_\la}

	\numberwithin{equation}{section}
	
\begin{document}
		\title{Pointwise Weyl Laws for Schr\"odinger operators with singular potentials}
	\author[]{Xiaoqi Huang}
\address[X.H.]{Department of Mathematics, University of Maryland, College Park, MD, 20742}
\email{xhuang49@umd.edu}

\author[]{Cheng Zhang}
\address[C.Z.]{Mathematical Sciences Center, Tsinghua University, Beijing, China}
\email{czhang98@tsinghua.edu.cn}
		\begin{abstract}
			We consider the Schr\"odinger operators $H_V=-\Delta_g+V$ with singular potentials $V$ on general $n$-dimensional Riemannian manifolds and study the eigenvalues and eigenfunctions under this perturbation. These singular  potentials  appear naturally in physics, most notably the Coulomb potential $|x|^{-1}$. Sogge and the first author \cite{hs} proved the sharp Weyl laws for these $H_V$ with potentials in the Kato class, which is the minimal assumption to ensure that $H_V$ is essentially self-adjoint and bounded from below and the eigenfunctions of $H_V$ are bounded. Later, Frank-Sabin \cite{fs} studied the problem on the pointwise Weyl laws for these $H_V$  in three dimensions by extending the method of Avakumovi\'c \cite{Avakumovic}, while it is unknown how to reconstruct this argument in other dimensions. In this paper, we completely solve this problem in any dimensions by using a different argument. First, we establish the pointwise Weyl law for potentials in the Kato class on any $n$-dimensional manifolds. This extends the 3-dimensional results of Frank-Sabin \cite{fs} by a different method. Second, we prove that the pointwise Weyl law with the standard sharp error term $O(\la^{n-1})$ holds for potentials in $L^n(M)$. This extends the classical results for smooth potentials by Avakumovi\'c \cite{Avakumovic}, Levitan \cite{Levitan} and H\"ormander \cite{HSpec} to critically singular potentials. In three dimensions, this $L^3$ condition also naturally appears in Boccato-Brennecke-Cenatiempo-Schlein \cite{bbcs} on the ground state energy of the Hamilton operator in the Gross-Pitaevskii regime. These two results are sharp, and our proof exploits Li-Yau's heat kernel bounds and Blair-Sire-Sogge's eigenfunction estimates.
		\end{abstract}
		\maketitle
		The purpose of this paper is to study the $ pointwise $ Weyl Law for the Schr\"odinger operators $H_V=-\Delta_g+V$ on compact $n$-dimensional $(n\ge2)$ Riemannian manifolds $(M,g)$ without boundary. We shall assume throughout that the potentials $V$ are real-valued. Moreover, we shall assume that $V\in \mathcal{K}(M)$, which is the Kato class. Recall that $\mathcal{K}(M)$ is all $V$ satisfying  
			\[\lim_{\delta\to 0}\sup_{x\in M}\int_{d_g(y,x)<\delta}|V(y)|W_n(d_g(x,y))dy=0,\]
		where \[W_n(r)=\begin{cases}r^{2-n},\quad\quad\quad\quad n\ge3\\
			\log(2+r^{-1}),\ \ n=2\end{cases}\]
		and $d_g$, $dy$ denote geodesic distance, the volume element on $(M,g)$.
		For later use, note that $L^{p}(M)\subset \mathcal{K}(M)\subset L^1(M)$ for all $p>\frac n2$. The Kato class $\mathcal{K}(M)$ and $L^{n/2}(M)$ share the same critical scaling
		behavior, while neither one is contained in the other one for $n\ge3$. For instance, singularities of the type $|x|^{-\alpha}$ for $\alpha<2$ are allowed for both classes. These singular  potentials  appear naturally in physics, most notably the Coulomb potential $|x|^{-1}$ in three dimensions. See e.g. Simon \cite{SimonSurvey} for a detailed introduction to the Schr\"odinger operators with potentials in the Kato class and their physical  motivations.
	
	As was shown in \cite{BSS} (see also \cite{SimonSurvey})
	the assumption that $V$ is in the Kato class is needed
	to ensure that  the Schr\"odinger operator $H_V$ is essentially self-adjoint and bounded from below, and the eigenfunctions of $H_V$ are bounded, which is an obvious requirement for a pointwise Weyl law to hold.  Although  the Schr\"odinger operators with potentials in $L^{n/2}(M)$ are also self-adjoint and bounded from below for $n\ge3$, the eigenfunctions for these potentials need not be bounded. Moreover, for potentials in the Kato class, the associated
	eigenfunctions are continuous by the heat kernel estimates of Li-Yau  \cite{Liyau} and Sturm \cite{sturm}.  Since $M$ is compact, the spectrum of $H_V$ is discrete. Assuming, as we may, that
	$H_V$ is a positive operator, we shall write the spectrum
	of $\sqrt{H_V}$  as
	\begin{equation}\label{1.5}
		\{\tau_k\}_{k=1}^\infty,
	\end{equation}
	where the eigenvalues, $\tau_1\le \tau_2\le \cdots$, are arranged in increasing order and we account for multiplicity.  For each $\tau_k$ there is an 
	eigenfunction $e_{\tau_k}\in \text{Dom }(H_V)$ (the domain of $H_V$) so that
	\begin{equation}\label{1.6}
		H_Ve_{\tau_k}=\tau^2_ke_{\tau_k}.
	\end{equation} 
	We shall always assume that the eigenfunctions are
	$L^2$-normalized, i.e.,
	$$\int_M |e_{\tau_k}(x)|^2 \, dx=1.$$
	
	After possibly adding a constant to $V$ we may, and shall, assume throughout that $H_V$ is bounded below by one, i.e.,
	\begin{equation}\label{1.7}
		\|f\|_2^2\le \langle \, H_Vf, \, f\, \rangle, \quad
		f\in \text{Dom }(H_V).
	\end{equation}
	Also, to be consistent, we shall let
	\begin{equation}\label{1.8}
		H^0=-\Delta_g
	\end{equation}
	be the unperturbed operator.  The corresponding eigenvalues and associated $L^2$-normalized
	eigenfunctions are denoted by $\{\lambda_j\}_{j=1}^\infty$ and $\{e^0_j\}_{j=1}^\infty$, respectively so
	that
	\begin{equation}\label{1.9}
		H^0e^0_j=\lambda^2_j e^0_j, \quad \text{and }\, \, 
		\int_M |e^0_j(x)|^2 \, dx=1.
	\end{equation}

	Both $\{e_{\tau_k}\}_{k=1}^\infty$ and $\{e^0_j\}_{j=1}^\infty$ are orthonormal bases for $L^2(M)$.  Recall (see e.g. \cite{SoggeHangzhou}) that if $N^0(\la)$ denotes the 
	Weyl counting function for $H^0$ then one has the 
	sharp Weyl law
	\begin{equation}\label{1.10}
		N^0(\la):=\# \{j: \, \la_j\le \la\}=(2\pi)^{-n} \omega_n \text{Vol}_g(M) \, \la^n
		\, +\, O(\la^{n-1}), 
	\end{equation}
	where $\omega_n=\pi^{\frac n2}/\Gamma(\frac n2+1)$ denotes the volume of the unit ball in
	$\Rn$ and $\text{Vol}_g(M)$ denotes the Riemannian volume of $M$.  This result is due to Avakumovi\'{c}~\cite{Avakumovic} and Levitan~\cite{Levitan}, and it was generalized to general self-adjoint elliptic pseudo-differential operators
	by H\"ormander~\cite{HSpec}.  The sharpness of \eqref{1.10} means that it
	cannot be improved for the standard sphere. The original Weyl law  was proved by Weyl \cite{weyl} for a compact domain in $\mathbb{R}^n$ over a hundred years ago. See Arendt, Nittka, Peter and Steiner \cite{anps} for historical background on this famous problem and its solution by Weyl.

	 Recall that 
	\begin{equation}\label{2.1'}
		N^0(\la):=\# \{j: \, \la_j\le \la\} =
		\int_M \sum_{\la_j\le \la} \, |e_{j}^0(x)|^2 \, dx.
	\end{equation}The  Weyl law \eqref{1.10} can be obtained from the following sharp pointwise Weyl law 
\begin{equation}\label{local}
	\sum_{\la_j\le\la}|e_j^0(x)|^2=(2\pi)^{-n}\omega_n\la^n+O(\la^{n-1}).
\end{equation}
It is due to Avakumovi\'{c}~\cite{Avakumovic}, following earlier partial results of Levitan~\cite{Levitan}, \cite{Levitan2}. The error term $O(\la^{n-1})$ is also sharp on the standard sphere. Proofs are presented in several texts, including H\"ormander~\cite{hbook} and Sogge \cite{fio}, \cite{SoggeHangzhou}. The pointwise Weyl law for a  compact domain in $\mathbb{R}^n$ is due to Carleman \cite{carl}. Similar results for compact manifolds with boundary are due to Seeley \cite{seeley1}, \cite{seeley2}.

	Recently, Huang-Sogge \cite{hs} proved that if $V\in\mathcal{K}(M)$, then the sharp Weyl law of the same form still holds for the Schr\"odinger operators $H_V$, i.e.
	\begin{equation}\label{1.12}
		N_V(\la):=\#\{k: \, \tau_k\le \la\}=(2\pi)^{-n}\omega_n \mathrm{Vol}_g(M) \, \la^n
		+O(\la^{n-1}).
	\end{equation}
See also \cite{sob}, \cite{fs}. Note that
\begin{equation}\label{2.1}
	N_V(\la):=\# \{k: \, \tau_k\le \la\} =
	\int_M \sum_{\tau_k\le \la} \, |e_{\tau_k}(x)|^2 \, dx.
\end{equation}
For $H_V$ with smooth potentials, the pointwise Weyl law of the form \eqref{local} follows from H\"ormander~\cite{HSpec}. So it is natural to study the pointwise Weyl law for $H_V$ with singular potentials.

Let $P^0=\sqrt{H^0}$ and $P_V=\sqrt{H_V}$. We denote the indicator function of the interval $[-\la,\la]$ by $\ola(\tau)$, and write 
\[\ola(P^0)(x,x)=\sum_{\la_j\le\la}|e_j^0(x)|^2,\ \ \ola(P_V)(x,x)=\sum_{\tau_k\le\la}|e_{\tau_k}(x)|^2.\]
When $n=3$, Frank-Sabin \cite{fs} proved that if $V\in \mathcal{K}(M)$, then as $\la\to\infty$ and uniformly in $x\in M$
 \begin{equation}\label{localfs}
 \ola(P_V)(x,x)=(2\pi)^{-3}\omega_3\la^3+o(\la^{3}).
 \end{equation}
They pointed out that the error term can not be replaced by $O(\la^{3-\delta})$, for any $\delta>0$. Moreover, \cite{fs} proved that if $V$ satisfies a stronger condition, the sharp pointwise law may hold. Indeed, if $V$ satisfies for some $\eps'>0$
\begin{equation}\label{stka}
	\sup_{x\in M}\int_{d_g(y,x)<\eps'}\frac{|V(y)|}{d_g(y,x)^2}dy<\infty,\end{equation}
then uniformly in $x\in M$
	\begin{equation}\label{sharplocal}\ola(P_V)(x,x)=(2\pi)^{-3}\omega_3\la^3+O(\la^{2}).\end{equation}
Note that  the condition \eqref{stka} is satisfied by $V\in L^{q}(M)$, for any $q>3$. For comparison, for any $q<3$, they showed that the sharp pointwise Weyl law \eqref{sharplocal} fails to hold for some $V\in L^q$. So $q=3$ is the threshold for the validity of the sharp pointwise Weyl law on the $L^q$ scale. The proof of \cite{fs} extends the method of Avakumovi\'c \cite{Avakumovic}, which relies on Tauberian theorems and parametrix estimates. To our knowledge, it is unknown how to reconstruct this argument in other dimensions, see \cite[Remark 4.5]{fs}. So it is an interesting open problem to determine the pointwise Weyl law for the Schr\"odinger operators with critically singular potentials on general $n$-dimensional manifolds.

In this paper, we completely solve this open problem in any dimensions. Our proof extends the wave equation method in \cite{hs}, \cite{SoggeHangzhou} to get around the difficulties in \cite{fs}.
	\begin{theorem}\label{thm}
			Let $n\ge2$ and $V\in \mathcal{K}(M)$. Then for any fixed $\eps>0$ there exists a $\Lambda(\eps,V)<\infty$ such that for $\lambda>\Lambda(\eps,V)$, we have
			\[\sup_{x\in M}|\ola(P_V)(x,x)-(2\pi)^{-n}\omega_n\la^n|\le C_V\eps\la^n.\]
			Here $C_V>0$ is a constant independent of $\la$ and $\eps$.
		\end{theorem}
So for potentials in the Kato class, as $\la\to\infty$ and uniformly in $x\in M$,
\begin{equation}\label{0.14}
		\sum_{\tau_k\le\la}|e_{\tau_k}(x)|^2=(2\pi)^{-n}\omega_n\la^n+o(\la^n).
		\end{equation}The 3-dimensional case is due to Frank-Sabin \cite{fs}, while other dimensions of Theorem \ref{thm} are new. 
			\begin{theorem}\label{thm2}
			Let $n\ge2$ and $V\in L^n(M)$. Then for $\la\ge1$
			\[\sup_{x\in M}|\ola(P_V)(x,x)-(2\pi)^{-n}\omega_n\la^n|\le C_V\la^{n-1}.\]
			Here $C_V>0$ is a constant independent of $\la$.
		\end{theorem}
		In other words, for potentials in $L^n(M)$, uniformly in $x\in M$,
		\begin{equation}\label{0.15}
		\sum_{\tau_k\le\la}|e_{\tau_k}(x)|^2=(2\pi)^{-n}\omega_n\la^n+O(\la^{n-1}).\end{equation} Theorem \ref{thm2} is new in all dimensions. In three dimensions, either \eqref{stka} or $L^3(M)$ can ensure the error term is $O(\la^{n-1})$, while neither of them can imply the other one. Moreover, it is worth mentioning that this $L^3$ condition also naturally appears in the study of the ground state energy of the Hamilton
	operator in the Gross-Pitaevskii regime, see Boccato-Brennecke-Cenatiempo-Schlein \cite{bbcs}.
	
These two theorems are sharp, by the explicit examples studied in the recent work \cite{hz2021} of the authors. Specifically, the sharpness of Theorem \ref{thm} means the error term $o(\la^n)$ cannot be replaced by $O(\la^{n-\delta})$ for any $\delta>0$. The sharpness of Theorem \ref{thm2} means that the condition $L^n(M)$ cannot be replaced by $L^{p}(M)$ for any $p<n$. In \cite{hz2021}, we consider the singular potentials
	 \[V(x)=\rho(d_g(x,x_0))d_g(x,x_0)^{-2+\eta},\ 0<\eta<1,\]where $x_0\in M$ is fixed, $d_g$ is the Riemannian distance function on $(M,g)$, and $\rho$ is a smooth cutoff function nonvanishing at zero. This $V$ is clearly in $\mathcal{K}(M)$, and it belongs to $L^q(M)$, for all $ q<\frac{n}{2-\eta}$. The pointwise Weyl law for $H_V$ with these $V$ is expected to have a sharp error term $\approx \la^{n-\eta}$.  In \cite{hz2021}, we proved this sharp bound on the flat torus $M=\mathbb{T}^n$ for any dimensions $n$. See also Frank-Sabin \cite{fs} for another different proof for the sharpness in three dimensions. Recall that $L^{p}(M)\subset \mathcal{K}(M)\subset L^1(M)$ for all $p>\frac n2$, and that the Kato class can ensure the boundedness of eigenfunctions, while $L^{\frac n2}$ cannot (see \cite{BSS}, \cite{SimonSurvey}). 
	 From the discussion above, we can see that Kato class is exactly the border for the existence of the pointwise Weyl law, and $p=n$ is the threshold for the validity of the sharp pointwise Weyl law (with the error term $O(\la^{n-1})$) on the $L^p$ scale.  If $V\in L^p(M)$ ($\frac n2<p< n$), one can easily modify the argument in the proof of Theorem \ref{thm2} to obtain the sharp error term $O(\la^{n-2+\frac np})$.

	 The main strategy of the proof is using Fourier analysis and the wave equation techniques to estimate the difference between the classical kernel $\ola(P^0)(x,x)$ and the Schr\"odinger kernel $\ola(P_V)(x,x)$. We reduce it to estimating the difference between their smooth approximations, by the Fourier inversion formula, Duhamel's principle and Sogge's $L^p$-spectral projection bounds. To deal with the difference of two kernels, the main difficulty is to handle the ``mixed terms'' with two kinds of frequencies from $P^0$ and $P_V$. To get around this, we must design new efficient frequency decompositions, and estimate the terms carefully by Li-Yau's heat kernel bounds and the theory of pseudo-differential operators. 
	 
	 It is likely that the sharp pointwise Weyl laws for $H_V$ can be improved under some global geometric conditions on the manifolds (see e.g. \cite{dg}, \cite{berard}, \cite{hla} for $V=0$). Moreover, it is interesting to investigate the Weyl laws for $H_V$ on compact manifolds with boundary (see Seeley \cite{seeley1}, \cite{seeley2} for $V=0$). We are working in progress on these problems. See also \cite{hsz}, \cite{bhss}, \cite{hs2}, \cite{hs3} for recent related works.
	 
	 The paper is organized as follow. In Section 1, we prove Theorem \ref{thm} by assuming Lemma \ref{hdiff} and Lemma \ref{chi}. In Secction 2, we prove Lemma \ref{hdiff}, and Lemma \ref{chi} follows by repeating the same argument. In Section 3, we prove Theorem \ref{thm2} by assuming Lemma \ref{hdiff'}. In Section 4, we prove Lemma \ref{hdiff'}. In the Appendix, we prove two lemmas used in the proof of Lemma \ref{hdiff'}. Throughout the paper, $A\ls B$ (or $A\gs B$) means $A\le CB$ (or $A\ge CB$) for some implicit constant $C>0$ that may change from line to line. $A\approx B$ means $A\ls B$ and $A\gs B$. All implicit constants $C$ are independent of the parameters $\la$, $\eps$, $\la_j$, $\tau_k$.
	 
	 \noindent\textbf{Acknowledgement. } The authors would like to thank Christopher Sogge and Allan Greenleaf for their helpful suggestions. C.Z. is partially supported by AMS-Simons Travel Grant.

		
		\section{Proof of Theorem \ref{thm}}
		Let $0<\eps<1$ and $\la\ge\eps^{-1}$. Fix an even real-valued function $\rho\in C^\infty(\mathbb{R})$ satisfying 
		\[\rho(t)=1\ \text{on}\ [-\delta_0/2,\delta_0/2]\ \text{and}\ \text{supp}\,\rho\subset(-\delta_0,\delta_0),\]
		where we assume that $\delta_0<\text{Inj}M$ (the injectivity radius of $M$).  For $\tau>0$, let 
		\begin{equation}\label{defin}
		h(\tau)=\frac1\pi\int \rho(\eps\la t)\frac{\sin \la t}{t}\cos t\tau dt.
		\end{equation}
		Then for $\tau>0$ \[|h(\tau)-\ola(\tau)|\ls(1+(\eps\la)^{-1}|\tau-\la|)^{-N},\ \forall N,\]
		and \begin{equation}\label{hdao}
			|\partial_\tau^j h(\tau)|\ls (\eps\la)^{-j}(1+(\eps\la)^{-1}|\tau-\la|)^{-N},\ \forall N,\ j=1,2,....
		\end{equation}
		Let us fix a non-negative function $\chi\in \mathcal{S}(\mathbb{R})$ satisfying:
		\[\chi(\tau)\ge1,\ |\tau|\le1,\ \text{and}\ \hat\chi(t)=0,\ |t|\ge\frac12.\]
		Let
		\begin{equation}\label{defchi}\tilde\chi_\la(\tau)=\frac{\eps\la}{\pi}\int \hat\chi(\eps\la t)e^{it\la}\cos\tau t dt=\chi((\eps\la)^{-1}(\la-\tau))+\chi((\eps\la)^{-1}(\la+\tau)).\end{equation}
		We have for $\tau>0$
		\begin{equation}\label{cdao}|\partial_{\tau}^j\tilde\chi_\la(\tau)|\ls (\eps\la)^{-j}(1+(\eps\la)^{-1}|\tau-\la|)^{-N}, \forall N,\ j=0,1,2,....\end{equation} 
	The key lemmas for Theorem \ref{thm} are the following.
		\begin{lemma}\label{hdiff}There exists a $\Lambda(\eps,V)<\infty$ such that for any $\la>\Lambda(\eps,V)$, we have \[\sup_{x\in M}|(h(P_V)-h(P^0))(x,x)|\ls \eps\la^n.\]
		\end{lemma}
		\begin{lemma}\label{chi}There exists a $\Lambda(\eps,V)<\infty$ such that for any $\la>\Lambda(\eps,V)$, we have
			\[\sup_{x\in M}|(\tilde\chi_\la(P_V)-\tilde\chi_\la(P^0))(x,x)|\ls \eps\la^n.\]
		\end{lemma}
		We postpone the proof of these lemmas. The following two lemmas will be used several times in the proof.
		\begin{lemma}[Spectral projection bounds, \cite{sogge88}]\label{sogge888}For $\la\ge1$, we have\[\|\1_{[\la,\la+1)}(P^0)\|_{L^2\to L^p}\ls \la^{\sigma(n,p)},\ 2\le p\le \infty,\]
			where $\sigma(n,p)=\max\{\frac{n-1}2(\frac12-\frac1p),\frac{n-1}2-\frac{n}{p}\}$.
		\end{lemma}
	These $L^p$-spectral projections bounds can be viewed as the generalized Tomas-Stein restriction estimates on manifolds. They were first obtained by Sogge \cite{sogge88}, and recently extended to the Schr\"odinger operators with critically singular potentials by Blair-Sire-Sogge \cite{BSS}. These bounds are sharp on $any$ closed manifolds. See \cite[Chapter 5]{fio}. 
		\begin{lemma}[Heat kernel bounds, \cite{Liyau}, \cite{sturm}]\label{heatk}
			If $V\in\mathcal{K}(M)$, then for $0<t\le 1$, there is a uniform constant $c=c_{M,V}>0$ so that
			\[e^{-tH_V}(x,y)\ls \begin{cases}t^{-n/2}e^{-cd_g(x,y)^2/t},\ \text{if}\  d_g(x,y)\le \text{Inj}(M)/2\\
				1,\ \ \text{otherwise}.
			\end{cases}\]
			Here Inj($M$) is the injectivity radius of $M$.
		\end{lemma}
The heat kernels bounds were first obtained by Li-Yau \cite{Liyau} for smooth potentials, and extended to the Kato class by Sturm \cite{sturm}.	Note that
	\[\sum_{\tau_k\le \la}|e_{\tau_k}(x)|^2\ls \sum_{\tau_\ell}e^{-\la^{-2}\tau_\ell^2}|e_{\tau_\ell}(x)|^2=e^{-\la^{-2}H_V}(x,x),\]
	so we have the following eigenfunction bounds.
		\begin{corr}[Rough eigenfunction bounds]\label{rough} If $V\in\mathcal{K}(M)$, then for $\la\ge1$
			\[\sup_{x \in M}\sum_{\tau_k\le \la}|e_{\tau_k}(x)|^2\le C_V\la^{n}.\]
		\end{corr}

		Using the classical pointwise Weyl Law for $P^0$ (see e.g. \cite{SoggeHangzhou}), we have
		
		\[h(P^0)(x,x)=\frac{\omega_n}{(2\pi)^{n}}\lambda^{n}+O(\eps\lambda^{n}).\]
		Then by Lemma \ref{hdiff}, we get
		\[h(P_V)(x,x)=\frac{\omega_n}{(2\pi)^{n}}\lambda^{n}+O(\eps\lambda^{n}).\]
		We claim that there exists a $\Lambda(\eps,V)<\infty$ such that for any $\la>\Lambda(\eps,V)$ we have\[\1_{[\la,\la+\eps\la)}(P_V)(x,x)=O(\eps\la^n).\]
		Indeed, it follows from
		\[\chi((\eps\la)^{-1}(\la-P_V))(x,x)=O(\eps\la^n).\]
	By \eqref{defchi}, it suffices to show
		\begin{equation}\label{tchi}\tilde\chi_\la(P_V)(x,x)=O(\eps\la^n),\end{equation}
		as  Corollary \ref{rough} implies \[\chi((\eps\la)^{-1}(\la+P_V))(x,x)=O(\eps^N\la^n),\ \forall N.\] By the classical $L^\infty$ spectral projection bounds (Lemma \ref{sogge888}), we know $\tilde\chi_\la(P^0)(x,x)=O(\eps\la^n)$, so we obtain \eqref{tchi}  by Lemma \ref{chi}. This proves the claim.
		
		As a corollary of this claim, we have for $j\ge0$ and $(1+\eps)^{-j-1}\la>\Lambda(\eps,V)$,
		\[\sum_{\tau_k\in[(1+\eps)^{-j-1}\la,(1+\eps)^{-j}\la)}|e_{\tau_k}(x)|^2=O(\eps((1+\eps)^{-j}\la)^n),\]
		and for $j\ge 0$ and $\la>\Lambda(\eps,V)$
		\[\sum_{\tau_k\in[(1+\eps)^{j}\la,(1+\eps)^{j+1}\la)}|e_{\tau_k}(x)|^2=O(\eps((1+\eps)^j\la)^n).\] Moreover, by Corollary \ref{rough},
		\[\sum_{\tau_k\in[1,\Lambda(\eps,V)]}|e_{\tau_k}(x)|^2\ls \Lambda(\eps,V)^n. \]Hence 	if $\la>\Lambda_1(\eps,V)$ with $\eps\Lambda_1^n>\Lambda^n$, then
		\begin{align*}\Big|(\ola(P_V)-&h(P_V))(x,x)\Big| \\
		&\le \sum_{k}|\ola(\tau_k)-h(\tau_k)||e_{\tau_k}(x)|^2\\
		&\ls\sum_{k}(1+(\eps\la)^{-1}|\tau_k-\la|)^{-N}|e_{\tau_k}(x)|^2\\
			&\ls \sum_{\tau_k\in[1,\Lambda(\eps,V)]}|e_{\tau_k}(x)|^2\\
			&+\sum_{(1+\eps)^{-j-1}\la>\Lambda(\eps,V)} \sum_{\tau_k\in[(1+\eps)^{-j-1}\la,(1+\eps)^{-j}\la)}(1+(\eps\la)^{-1}|(1+\eps)^{-j}\la-\la|)^{-N}|e_{\tau_k}(x)|^2\\
			&+ \sum_{j=0}^\infty (1+(\eps\la)^{-1}|(1+\eps)^j\la-\la|)^{-N}\sum_{\tau_k\in[(1+\eps)^{j}\la,(1+\eps)^{j+1}\la)}|e_{\tau_k}(x)|^2\\
			&\ls \Lambda(\eps,V)^n+\sum_{j=0}^\infty (1+(\eps\la)^{-1}|(1+\eps)^{-j}\la-\la|)^{-N}\eps((1+\eps)^{-j}\la)^n\\
			&+ \sum_{j=0}^\infty (1+(\eps\la)^{-1}|(1+\eps)^j\la-\la|)^{-N}\eps((1+\eps)^j\la)^n\\
			&\ls \eps\lambda^{n}.\end{align*}
		So for $\la>\Lambda_1(\eps,V)$, \[\ola(P_V)(x,x)=\frac{\omega_n}{(2\pi)^n}\la^n+O(\eps\la^n).\]
		To complete the proof, we only need to prove Lemma \ref{hdiff}, and Lemma \ref{chi}. For simplicity, we shall only give the proof of Lemma \ref{hdiff} here, since $\tilde\chi_\la(\tau)$  satisfies \eqref{cdao} that is analogous to the estimate \eqref{hdao} of $h(\tau)$, Lemma \ref{chi} follows by repeating the same argument.
		
		\section{Proof of Lemma \ref{hdiff}}
		First, we follow the reduction argument in \cite{hs}. Let \[\cos tP^0(x,y)=\sum_{j}\cos t\la_j e_j^0(x)e_j^0(y).\]
		It is the kernel of the solution operator for $f\to (\cos tP^0)f=u^0(t,x)$, where $u^0(t,x)$ solves the wave equation
		\[\left(\partial_{t}^{2}+H^{0}\right) u^{0}(x, t)=0,(x, t) \in M \times \mathbb{R},\left.u^{0}\right|_{t=0}=f,\left.\partial_{t} u^{0}\right|_{t=0}=0.\] Similarly, 
		\[\bigl(\cos(tP_V)\bigr)(x,y)=\sum_k \cos t\tau_k \, e_{\tau_k}(x)e_{\tau_k}(y)\]is the kernel of $f\to \cos (tP_V)f=u_V(x,t)$, where 
		$u_V$ solves the wave equation
\[(\partial_t^2+H_V)u_V(x,t)=0, \, \, (x,t)\in M\times \R,
			\, \, u_V|_{t=0}=f, \, \, \partial_tu_V|_{t=0}=0.\] By Duhamel's principle,
		\begin{align*}\cos tP_V(x,y)&-\cos tP^0(x,y)\\
	&=-\sum_{\la_j}\sum_{\tau_k}\int_M\int_0^t\frac{\sin(t-s)\la_j}{\la_j}\cos s\tau_k\  e_j^0(x)e_j^0(z)e_{\tau_k}(z)e_{\tau_k}(y)V(z)dzds\\
			&=\sum_{\la_j}\sum_{\tau_k}\int_M\frac{\cos t\la_j-\cos t\tau_k}{\la_j^2-\tau_k^2} e_j^0(x)e_j^0(z)e_{\tau_k}(z)e_{\tau_k}(y)V(z)dz.\end{align*}
		Recall that
		\[h(\tau)=\frac1\pi\int \rho(\eps\la t)\frac{\sin \la t}{t}\cos t\tau dt.\]
		We have
		\[h(P_V)(x,y)-h(P^0)(x,y)=\sum_{\la_j}\sum_{\tau_k}\int_M\frac{h(\la_j)-h(\tau_k)}{\la_j^2-\tau_k^2} e_j^0(x)e_j^0(z)e_{\tau_k}(z)e_{\tau_k}(y)V(z)dz\]
		So we just need to prove there exists a $\Lambda(\eps,V)>0$ such that for any $\la>\Lambda(\eps,V)$ we have
		\begin{equation}\label{main}
		\Bigl| \, \sum_{\lambda_j}
		\sum_{\tau_k}\int_M
		\frac{h(\la_j)-h(\tau_k)}{\la_j^2-\tau_k^2} e_j^0(x)e_j^0(y)e_{\tau_k}(x)
		e_{\tau_k}(y)V(y) \, dy\, \Bigr|\ls\eps\la^n.
		\end{equation}
		Decompose the sum into two parts with respect to the frequencies $\tau_k$: \[\sum_{\lambda_j}
		\sum_{\tau_k}=\sum_{\tau_k\le10\la}\sum_{\la_j}+\sum_{\tau_k>10\la}\sum_{\la_j}\]
		We will use the following lemma several times.
	\begin{lemma}[Kernel estimates of PDO]\label{pdo}
		Let $\mu\in \mathbb{R}$, and $m\in C^\infty(\mathbb{R})$ belong to the symbol class $S^\mu$, that is, assume that
		\[\Big|\Big(\frac{d}{d t}\Big)^{\alpha} m(t)\Big| \leq C_{\alpha}(1+|t|)^{\mu-\alpha} \quad \forall \alpha.\]
		Then $m(P^0)$ is a pseudo-differential operator of order $\mu$. Moreover, if $R\ge1$, then the kernel of the operator $m(P^0/R)$ satisfies for all $N\in\mathbb{N}$
		\begin{equation}\label{withla}
			|m(P^0/R)(x,y)|\\\le  \begin{cases}CR^n\big(R d_g(x,y)\big)^{-n-\mu}\big(1+R d_g(x,y)\big )^{-N},\,\quad\quad\,\,\,\, \ n+\mu>0\\
				CR^n\log(2+(Rd_g(x,y))^{-1})\big(1+R d_g(x,y)\big)^{-N},\,\, \ n+\mu=0\\
				CR^n(1+Rd_g(x,y))^{-N},\ \ \ \ \ \quad\quad\quad\quad\quad\quad\quad\quad\ \  n+\mu<0.
			\end{cases}
		\end{equation}
	\end{lemma}

See \cite[Theorem 4.3.1]{fio}, \cite[Prop.1 on page 241]{steinbook} for the proof.  In the lemma, we mean that the inequalities hold near the diagonal (so that $d_g(x,y)$ is well-defined) and that outside the neighborhood of the diagonal we have $|m(P^0/R)(x,y)|\ls R^{-N}$ for all $N$.

 We first deal with the low-frequency part $\tau_k\le 10\la$.
 
\noindent \textbf{1. Low-frequency $(\tau_k\le 10\la)$.}
		
		Let \[M_0: \la_j\mapsto\frac{h(\la_j)-h(\tau_k)}{\la_j^2-\tau_k^2}\] and $P_\ell(\tau):=h(\tau_k)-\sum_{j=0}^\ell\frac1{j!}h^{(j)}(\tau)(\tau_k-\tau)^j$ for $\ell=0,1,2,...$. Then by induction we have
		\begin{equation}\label{dao}
			\partial_\tau^\ell \Big(\frac{h(\tau)-h(\tau_k)}{\tau-\tau_k}\Big)=\frac{\ell!}{(\tau_k-\tau)^{\ell+1}}P_\ell(\tau),\ \ell=0,1,2,....
		\end{equation}
		Since \begin{equation}\label{pell}
			P_\ell(\tau)=\frac1{(\ell+1)!}h^{(\ell+1)}(\tau+\theta(\tau_k-\tau))(\tau_k-\tau)^{\ell+1}
		\end{equation} for some $0\le \theta\le 1$, by \eqref{hdao} we have for $|\tau-\tau_k|\le \frac12\tau$,
		\begin{equation}\label{eq1}
			|\partial_\tau^\ell \Big(\frac{h(\tau)-h(\tau_k)}{\tau-\tau_k}\Big)|\approx |h^{(\ell+1)}(\tau+\theta(\tau_k-\tau))|\ls \eps^{-1-\ell}(1+\tau)^{-1-\ell}.
		\end{equation}
		So when $|\la_j-\tau_k|\le \frac12\la_j$, $\la_j+\theta(\tau_k-\la_j)\approx \la_j$, by \eqref{hdao}, we have for $\ell=0,1,2,....$,
		\begin{equation}\label{eq2}
			|\partial_{\la_j}^\ell M_0(\la_j)|\ls\begin{cases} \eps^{-1-\ell}(1+\la_j)^{-2-\ell}, \,\,\,\text{if}\,\,\, \la_j \ge 2\la \\ \eps^{-1-\ell}(1+\la)^{-2-\ell}, \,\,\,\text{if}\,\,\, \la_j \le 2\la
			\end{cases}
		\end{equation}
		When $|\la_j-\tau_k|> \frac12\la_j$, it follows directly from \eqref{hdao} that
		\begin{equation}\label{eq3}
			|\partial_{\la_j}^\ell M_0(\la_j)|\ls \begin{cases} \eps^{-1-\ell}(1+\la_j)^{-2-\ell}, \,\,\,\text{if}\,\,\, \la_j > 2\tau_k \\ \eps^{-1-\ell}(1+\tau_k)^{-2-\ell}, \,\,\,\text{if}\,\,\, \la_j < \frac23\tau_k
			\end{cases}
		\end{equation}
		The implicit constants in \eqref{eq1}, \eqref{eq2}, \eqref{eq3} are independent of $\tau_k$. Now we have shown $M_0\in S^{-2}$, and satisfies
		\begin{equation}\label{eq4}
			|\partial_{\la_j}^\ell M_0(\la_j)|\ls \eps^{-1-\ell}(1+\la_j)^{-2-\ell},\ \ell=0,1,2,....
		\end{equation}
		Then by Lemma \ref{pdo}, we have for some constant $n_0>0$ (only dependent on $n$),
		\begin{equation}\label{eq5}
		|\sum_{\la_j}\frac{h(\la_j)-h(\tau_k)}{\la_j^2-\tau_k^2} e_j^0(x)e_j^0(y)|\ls \eps^{-n_0}W_n(d_g(x,y)).
		\end{equation}
		Additionally, by \eqref{eq2} and \eqref{eq3}, it is straightforward to check that for $\tau_k\le 10\la$, if we let $\tilde M_0(\mu)=\tau_k^2M_0(\tau_k\mu)$, then
			\begin{equation}\label{eq6}
			|\partial_{\mu}^\ell \tilde M_0(\mu)|\ls \eps^{-1-\ell}(1+\mu)^{-2-\ell},\ \ell=0,1,2,....
		\end{equation}
		Thus, using \eqref{withla}, since $M_0(\mu)=\tau_k^{-2}\tilde M_0(\mu/\tau_k)$, we also have
		\begin{equation}\label{eq7}
		|\sum_{\la_j}\frac{h(\la_j)-h(\tau_k)}{\la_j^2-\tau_k^2} e_j^0(x)e_j^0(y)|\ls \eps^{-n_0}\tau_k^{n-2}W_n(\tau_k d_g(x,y))\big(1+\tau_k d_g(x,y)\big)^{-\sigma},\ \forall \sigma.
		\end{equation}
		We may write $V=V_{\le N}+ V_{>N}$, where \[V_{\le N}(x)=\begin{cases}V(x),\ \text{if}\ |V(x)|\le N,\\
			0,\ \text{otherwise.}
			\end{cases}\]One the one hand, by the fact that $V\in\mathcal{K}(M)\subset L^1(M)$ we can choose $N(\eps,V)<\infty$ and $\delta(\eps,V)>0$  such that
		\[\int_{d_g(x,y)\le\delta}|V_{>N}(y)|W_n(d_g(x,y))dy<\eps^{n_0+1},\]
		\[\int_{d_g(x,y)>\delta} |V_{>N}(y)|W_n(d_g(x,y))dy<\eps^{n_0+1}.\]
		Recall that Corollary \ref{rough}  gives $\sum_{\tau_k\le 10\la}|e_{\tau_k}(x)e_{\tau_k}(y)|\ls \la^n$. So by using \eqref{eq5} we get 
		\begin{equation}
		|\sum_{\tau_k\le 10\la}\sum_{\la_j} \int_M
		\frac{h(\la_j)-h(\tau_k)}{\la_j^2-\tau_k^2} e_j^0(x)e_j^0(y)e_{\tau_k}(x)
		e_{\tau_k}(y)V_{>N}(y)dy|\ls \eps\la^n.
		\end{equation}
		On the other hand, since Corollary \ref{rough}  also gives  $\sum_{\tau_k\approx 2^m}|e_{\tau_k}(x)e_{\tau_k}(y)|\ls 2^{mn}$, by using \eqref{eq7} we have
		\begin{multline}\label{eq8}
		|\sum_{\tau_k\le 10\la}\sum_{\la_j} \int_M
		\frac{h(\la_j)-h(\tau_k)}{\la_j^2-\tau_k^2} e_j^0(x)e_j^0(y)e_{\tau_k}(x)
		e_{\tau_k}(y)V_{\le N}(y)dy| \\ \ls \sum_{m\in \mathbb{N}: 2^m\le10\la} N \eps^{-n_0}2^{mn}\cdot2^{-2m} \ls N \eps^{-n_0}\la^{n-2}\log\la.
		\end{multline}
		Thus the proof of \eqref{main} is complete for the low-frequency part $\tau_k\le10\la$ if we choose $\Lambda(\e, V)$ such that $N \eps^{-n_0}\Lambda^{n-2}\log\Lambda\le\e \Lambda^{n}$.
		
		Next, we only need to deal with high-frequency part $\tau_k>10\la$.
		
\noindent \textbf{2. High-frequency ($\tau_k>10\la)$.}

		Choose smooth cut-off functions such that
		\[\beta(\la_j\ls \la)+\sum_{2^m\ge5\la}\beta(\la_j\approx 2^m)=1,\ \la_j\ge1,\]
		\[\beta(\tau_k\ls 2^m)+\beta(\tau_k\gs 2^m)=1,\ \text{if  }\ 2^m\ge5\la.\]
		Here $\beta(\la_j\ls \la)$ is supported on $\{\la_j< 5\la\}$, and $\beta(\la_j\approx 2^m)$ is supported on $\{2^{m-2}<\la_j<2^m\}$, and $\beta(\tau_k\gs 2^m)$ is supported on $\{\tau_k>2^{m+1}\}$. Thus, if $2^m\ge5\la$ and $\tau_k>10\la$, then $\tau_k> 2\la_j$ on the support of $\beta(\la_j\approx 2^m)\beta(\tau_k\gs 2^m)$. If  $\tau_k>10\la$, then $\tau_k>2\la_j$ on the support of $\beta(\la_j\ls \la)$.

		For each fixed $\tau_k>10\la$, let $$M(\tau)=\frac{h(\tau)-h(\tau_k)}{\tau^2-\tau_k^2},$$ where $h(\tau)$ is defined as in \eqref{defin}. We shall decompose the sum into three parts, 
		
		\begin{equation}\label{eq1'}\int_M\sum_{2^m\ge5\la}\sum_{\la_j}\sum_{\tau_k>10\la}M(\la_j)\beta(\la_j\approx 2^m)\beta(\tau_k\ls 2^m)e_j^0(x)e_j^0(y)e_{\tau_k}(x)e_{\tau_k}(y)V(y)dy,\end{equation}
		\begin{equation}\label{eq2'}\int_M\sum_{2^m\ge5\la}\sum_{\la_j}\sum_{\tau_k>10\la}M(\la_j)\beta(\la_j\approx 2^m)\beta(\tau_k\gs 2^m)e_j^0(x)e_j^0(y)e_{\tau_k}(x)e_{\tau_k}(y)V(y)dy,\end{equation}
		\begin{equation}\label{eq3'}
			\int_M\sum_{\la_j}\sum_{\tau_k>10\la}M(\la_j)\beta(\la_j\ls \la)e_j^0(x)e_j^0(y)e_{\tau_k}(x)e_{\tau_k}(y)V(y)dy.
		\end{equation}
		In the following three subsections, we show that they are all $O(\eps\la^n)$ for large $\la$.
		
	\noindent	\textbf{1. Estimate of \eqref{eq1'}.}
		First, by \eqref{hdao} and mean value theorem, it is not hard to see that \[M_1:\ \la_j\mapsto \sum_{2^m\ge5\la}\frac{h(\la_j)-h(\tau_k)}{\la_j^2-\tau_k^2}\beta(\la_j\approx 2^m)\beta(\tau_k\ls 2^m)\] is a symbol in $S^{-2}$, and satisfies 
		\[|\partial_{\la_j}^\ell M_1(\la_j)|\ls (1+\la_j)^{-2-\ell}(1+(\eps\la)^{-1}\tau_k)^{-N},\ \forall  N,\ \ell=0,1,2,.... ,\]
		where we used the fact that $|\tau_k-\la|\approx \tau_k$ if $\tau_k>10\la$.
	Then by Lemma \ref{pdo}, the kernel satisfies
		\[|M_1(P^0)(x,y)|\ls W_n(d_g(x,y))(1+(\eps\la)^{-1}\tau_k)^{-N}, \ \ \forall  N\]
		and  Corollary \ref{rough} gives
	\[\sum_{\tau_k>10\la}(1+(\eps\la)^{-1}\tau_k)^{-N}|e_{\tau_k}(x)e_{\tau_k}(y)|\ls \sum_{\ell\in\mathbb{N}:2^\ell>10\la}(1+(\eps\la)^{-1}2^\ell)^{-N}2^{n\ell}\ls\eps^N\la^n\le\eps\la^n.\]
		Thus using $V\in \mathcal{K}(M)$ we get
		\[|\int_M\sum_{2^m\ge5\la}\sum_{\la_j}\sum_{\tau_k>10\la}M_1(\la_j)e_j^0(x)e_j^0(y)e_{\tau_k}(x)e_{\tau_k}(y)V(y)dy|\ls \eps\la^n.\]
		
		\noindent \textbf{2. Estimate of \eqref{eq2'}.}
		Second, by our construction of the cut-off functions, we have  $\tau_k> 2\la_j$ on the support of $\beta(\la_j\approx 2^m)\beta(\tau_k\gs 2^m)$ if $2^m\ge5\la$ and $\tau_k>10\la$. Then
		\[M_{21}: \la_j\mapsto \sum_{2^m\ge5\la}\frac{h(\tau_k)}{\la_j^2-\tau_k^2}\beta(\la_j\approx 2^m)\beta(\tau_k\gs 2^m)\]
		is a symbol in $S^{-2}$ and satisfies
		\[|\partial_{\la_j}^\ell M_{21}(\la_j)|\ls (1+\la_j)^{-2-\ell}(1+(\eps\la)^{-1}\tau_k)^{-N},\ \forall  N,\ \ell=0,1,2,....\]
		Hence by Lemma \ref{pdo}, the kernel satisfies
		\[|M_{21}(P^0)(x,y)|\ls W_n(d_g(x,y))(1+(\eps\la)^{-1}\tau_k)^{-N},\]
		and again  Corollary \ref{rough} gives
		\[\sum_{\tau_k>10\la}(1+(\eps\la)^{-1}\tau_k)^{-N}|e_{\tau_k}(x)e_{\tau_k}(y)|\ls \eps^N\la^n\le\eps\la^n.\]
		Thus using $V\in \mathcal{K}(M)$ we get
		\[|\int_M\sum_{2^m\ge5\la}\sum_{\la_j}\sum_{\tau_k>10\la}M_{21}(\la_j)e_j^0(x)e_j^0(y)e_{\tau_k}(x)e_{\tau_k}(y)V(y)dy|\ls \eps\la^n.\]
		Moreover, if we write
		\[\frac{-h(\la_j)}{\la_j^2-\tau_k^2}=\int_0^\infty h(\la_j)e^{t(\la_j^2-\tau_k^2)}dt=\int_0^{2^{-2m}} h(\la_j)e^{t(\la_j^2-\tau_k^2)}dt-\frac{h(\la_j)e^{2^{-2m}(\la_j^2-\tau_k^2)}}{\la_j^2-\tau_k^2},\]
		then
		\[M_{22}:\la_j\mapsto  \frac{h(\la_j)e^{2^{-2m}\la_j^2}}{\la_j^2-\tau_k^2}\beta(\la_j\approx 2^m)\beta(\tau_k\gs 2^m)\]
		is a symbol in $S^{-2}$ and satisfies
		\[|\partial_{\la_j}^\ell M_{22}(\la_j)|\ls (1+\la_j)^{-2-\ell}(1+(\eps\la)^{-1}2^m)^{-N},\ \forall  N,\ \ell=0,1,2,....\]
		Hence by Lemma \ref{pdo}, the kernel satisfies
		\[|M_{22}(P^0)(x,y)|\ls W_n(d_g(x,y))(1+(\eps\la)^{-1}2^m)^{-N}.\]
		By Corollary \ref{rough},
		\[\sum_{\tau_k\gs 2^m}e^{-2^{-2m}\tau_k^2}|e_{\tau_k}(x)e_{\tau_k}(y)|\ls \sum_{\ell\in\mathbb{N}:2^\ell\gs2^m}e^{-2^{-2m}2^{2\ell}}2^{n\ell}\ls 2^{nm}.\]
		Thus using $V\in \mathcal{K}(M)$ we get
		\begin{align*}
			|\int_M\sum_{2^m\ge5\la}\sum_{\la_j}\sum_{\tau_k>10\la}&M_{22}(\la_j)e_j^0(x)e_j^0(y)e_{\tau_k}(x)e_{\tau_k}(y)V(y)|\\
			&\ls \sum_{2^m\ge5\la}2^{nm}(1+(\eps\la)^{-1}2^m)^{-N}\\
			&\ls \eps^N\la^n\le\eps\la^n.
		\end{align*}
		Furthermore, by heat kernel bounds Lemma \ref{heatk} and $\la_j\approx 2^m$
		\begin{align}|\int_0^{2^{-2m}}e^{t\la_j^2}&\sum_{\tau_k}e^{-t\tau_k^2}e_{\tau_k}(x)e_{\tau_k}(y)dt|\\ \nonumber
			&\ls \int_0^{2^{-2m}}|\sum_{\tau_k}e^{-t\tau_k^2}e_{\tau_k}(x)e_{\tau_k}(y)|dt\\ \nonumber
			&\ls \int_0^{2^{-2m}} t^{-\frac n2}e^{-cd_g(x,y)^2/t}dt\\ \nonumber
			&\ls \begin{cases}
				\log(2+(2^m d_g(x,y))^{-1}),\quad n=2\\\nonumber
				d_g(x,y)^{2-n},\quad\quad\quad\quad\quad\quad\  n\ge3
			\end{cases}\\ \nonumber
			&\ls W_n(d_g(x,y)).\end{align}
		Moreover, by Corollary \ref{rough}
		\[\sum_{\la_j}|h(\la_j)\beta(\la_j\approx 2^m)||e_j^0(x)e_j^0(y)|\ls 2^{nm}(1+(\eps\la)^{-1}2^m)^{-N},\ \forall N.\]
		Thus using $V\in \mathcal{K}(M)$ we get
		\begin{align*}
			|\int_M\sum_{2^m\ge5\la}\sum_{\la_j}\sum_{\tau_k}&\int_0^{2^{-2m}} h(\la_j)e^{t(\la_j^2-\tau_k^2)}dt\beta(\la_j\approx 2^m)e_j^0(x)e_j^0(y)e_{\tau_k}(x)e_{\tau_k}(y)V(y)dy|\\
			&\ls \sum_{2^m\ge5\la}2^{nm}(1+(\eps\la)^{-1}2^m)^{-N}\\
			&\ls \eps^N\la^n\le\eps\la^n.
		\end{align*}
		Since $1=\beta(\tau_k\ls 2^m)+\beta(\tau_k\gs 2^m)$ when $2^m\ge 5\la$,  it remains to estimate
		\[ \int_M\sum_{2^m\ge5\la}\sum_{\la_j}\sum_{\tau_k}\int_0^{2^{-2m}} M_{23}(\la_j)e_j^0(x)e_j^0(y)e_{\tau_k}(x)e_{\tau_k}(y)V(y)dy,\]
		where 
		$$ M_{23}(\la_j)=\int_0^{2^{-2m}} h(\la_j)e^{t(\la_j^2-\tau_k^2)}\beta(\la_j\approx 2^m)\beta(\tau_k\ls 2^m) \,dt
		$$
		Indeed, $M_{23}$
		is a symbol in $S^{-2}$ which satisfies
		\[|\partial_{\la_j}^\ell M_{23}(\la_j)|\ls (1+\la_j)^{-2-\ell}(1+(\eps\la)^{-1}2^m)^{-N},\ \forall  N,\ \ell=0,1,2,....\]
		So by Lemma \ref{pdo}, the kernel satisfies
		\[|M_{23}(P^0)(x,y)|\ls W_n(d_g(x,y))(1+(\eps\la)^{-1}2^m)^{-N}.\]
		By Corollary \ref{rough}
		\[\sum_{\tau_k\ls 2^m}|e_{\tau_k}(x)e_{\tau_k}(y)|\ls 2^{mn}.\]
		Thus
		\begin{align*}
			|\int_M\sum_{2^m\ge5\la}\sum_{\la_j}\sum_{\tau_k}\int_0^{2^{-2m}} &M_{23}(\la_j)e_j^0(x)e_j^0(y)e_{\tau_k}(x)e_{\tau_k}(y)V(y)dy|\\
			&\ls \sum_{2^m\ge 5\la}2^{mn}(1+(\eps\la)^{-1}2^m)^{-N}\\
			&\ls \eps^N\la^n\le\eps\la^n.
		\end{align*}

		\noindent \textbf{3. Estimate of \eqref{eq3'}.}
		By our construction of the cut-off functions, we have $\tau_k>2\la_j$ on the support of $\beta(\la_j\ls \la)$ if $\tau_k>10\la$, so
		\[M_{31}: \la_j\mapsto \frac{h(\tau_k)}{\la_j^2-\tau_k^2}\beta(\la_j\ls\la)\]
		is a symbol in $S^{-2}$ and satisfies
		\[|\partial_{\la_j}^\ell M_{31}(\la_j)|\ls (1+\la_j)^{-2-\ell}(1+(\eps\la)^{-1}\tau_k)^{-N},\ \forall  N,\ \ell=0,1,2,....\]
		Then by Lemma \ref{pdo} the kernel satisfies
		\[|M_{31}(P^0)(x,y)|\ls W_n(d_g(x,y))(1+(\eps\la)^{-1}\tau_k)^{-N},\]
		and again Corollary \ref{rough} implies
		\[\sum_{\tau_k>10\la}(1+(\eps\la)^{-1}\tau_k)^{-N}|e_{\tau_k}(x)e_{\tau_k}(y)|\ls\eps^N\la^n\le\eps\la^n.\]
		Thus using $V\in \mathcal{K}(M)$ we get
		\[|\int_M\sum_{\la_j}\sum_{\tau_k>10\la}M_{31}(\la_j)e_j^0(x)e_j^0(y)e_{\tau_k}(x)e_{\tau_k}(y)V(y)dy|\ls \eps\la^n.\]
		Moreover, we write
		\begin{align*}\frac{-h(\la_j)}{\la_j^2-\tau_k^2}\1(\tau_k>10\la)\beta(\la_j\ls\la)&=\int_0^\infty h(\la_j)e^{t(\la_j^2-\tau_k^2)}dt\1(\tau_k>10\la)\beta(\la_j\ls\la)\\
			&=\int_0^{\la^{-2}} h(\la_j)e^{t(\la_j^2-\tau_k^2)}dt  \beta(\la_j\ls\la)\\ &\quad-\int_0^{\la^{-2}} h(\la_j)e^{t(\la_j^2-\tau_k^2)}dt\1(\tau_k\le10\la)\beta(\la_j\ls\la)\\ &\quad -\frac{h(\la_j)e^{\la^{-2}(\la_j^2-\tau_k^2)}}{\la_j^2-\tau_k^2}\1(\tau_k>10\la)\beta(\la_j\ls\la)\\
			&:=M_{34}(\la_j)-M_{33}(\la_j)-M_{32}(\la_j).\end{align*}
		In the following, we handle these three parts separately.
		First, let
		\[\tilde M_{32}:\mu\mapsto  \frac{h(\la\mu)e^{-\mu^2}}{\mu^2-(\tau_k/\la)^2}\beta(\mu\ls1)\1(\tau_k>10\la).\]
		It is a symbol in $S^{-2}$ and satisfies
		\[|\partial_{\mu}^\ell \tilde M_{32}(\mu)|\ls \eps^{-\ell}(1+\mu)^{-2-\ell},\ \ell=0,1,2,....\]
		Here $\beta(\mu\ls 1)$ is supported on $\{\mu\le5\}$.
		Thus by Lemma \ref{pdo} the kernel satisfies for some constant $n_0>0$ (only dependent on $n$)
		\begin{align}\label{ker01}|M_{32}(P^0)(x,y)|&=\la^{-2}|\tilde M_{32}(P^0/\la)(x,y)|\\ \nonumber
	&\ls \begin{cases}
		\eps^{-n_0}\log(2+(\la d_g(x,y))^{-1})(1+\la d_g(x,y))^{-N},\ n=2\\\nonumber
		\eps^{-n_0}d_g(x,y)^{2-n}(1+\la d_g(x,y))^{-N},\quad\quad\quad\quad\quad n\ge3
	\end{cases}\\ \nonumber
			&\ls 
			\eps^{-n_0}W_n(d_g(x,y))(1+\la d_g(x,y))^{-N},\ \forall N.\end{align}
		By Corollary \ref{rough},
		\[\sum_{\tau_k>10\la}e^{-\la^{-2}\tau_k^2}|e_{\tau_k}(x)e_{\tau_k}(y)|\ls \sum_{\ell\in\mathbb{N}:2^\ell>10\la}e^{-\la^{-2}2^{2\ell}}2^{n\ell}\ls\la^n.\]
		Fix $N=1$ in \eqref{ker01}. Thus, using \eqref{ker01} and $V\in \mathcal{K}(M)$ we may choose $\Lambda(\eps,V)<\infty$ so that for $\la>\Lambda(\eps,V)$
		\begin{align}\label{m32}
			\int_M& \la^{-2}|M_{32}(P^0/\la)(x,y)V(y)|dy\\ \nonumber
			&\ls \eps^{-n_0}\int_{d_g(x,y)\le\la^{-1/2}}W_n(d_g(x,y))|V(y)|dy+\eps^{-n_0}\la^{-1/2}\int_MW_n(d_g(x,y))|V(y)|dy\\ \nonumber
			&\ls\eps. \nonumber
		\end{align}
		So we have
		\begin{align}\label{spec01}
			|\int_M\sum_{\la_j}\sum_{\tau_k>10\la}&M_{32}(\la_j)e_j^0(x)e_j^0(y)e_{\tau_k}(x)e_{\tau_k}(y)V(y)dy|\ls \eps\la^n.
		\end{align}
		Second, if $\tau_k\le10\la$, then
		\[\tilde M_{33}:\mu\mapsto \int_0^{1} h(\la\mu)e^{t(\mu^2-(\tau_k/\la)^2)}dt\beta(\mu\ls1)\1(\tau_k\le10\la)\]
		is a symbol in $S^{-2}$ and satisfies
		\[|\partial_{\mu}^\ell \tilde M_{33}(\mu)|\ls \eps^{-\ell}(1+\mu)^{-2-\ell},\ \ell=0,1,2,....\]
		So by Lemma \ref{pdo}, the kernel satisfies for some constant $n_0>0$ (only dependent on $n$)
		\begin{align}\label{ker03}|M_{33}(P^0)(x,y)|&=\la^{-2}|\tilde M_{33}(P^0/\la)(x,y)|\\ \nonumber
			&\ls \begin{cases}
				\eps^{-n_0}\log(2+(\la d_g(x,y))^{-1})(1+\la d_g(x,y))^{-N},\ n=2\\\nonumber
				\eps^{-n_0}d_g(x,y)^{2-n}(1+\la d_g(x,y))^{-N},\quad\quad\quad\quad\quad n\ge3
			\end{cases}\\ \nonumber
			&\ls 
			\eps^{-n_0}W_n(d_g(x,y))(1+\la d_g(x,y))^{-N},\ \forall N.\end{align}
		By Corollary \ref{rough}
		\[\sum_{\tau_k\le 10\la}|e_{\tau_k}(x)e_{\tau_k}(y)|\ls \la^n.\]
		Thus, as in \eqref{m32}, using \eqref{ker03} and $V\in\mathcal{K}(M)$ we may choose $\Lambda(\eps,V)<\infty$ so that for $\la>\Lambda(\eps,V)$
		\begin{align}\label{spec03}
			|\int_M\sum_{\la_j}\sum_{\tau_k\le10\la} M_{33}(\la_j)e_j^0(x)e_j^0(y)e_{\tau_k}(x)e_{\tau_k}(y)V(y)dy|\ls \eps\la^n.
		\end{align}
		Third, by heat kernel bounds Lemma \ref{heatk} and $\la_j\ls\la$
		\begin{align}\label{ker02} |\int_0^{\la^{-2}}e^{t\la_j^{2}}\sum_{\tau_k}e^{-t\tau_k^2}&e_{\tau_k}(x)e_{\tau_k}(y)dt|\\ \nonumber
		&\ls \int_0^{\la^{-2}}|\sum_{\tau_k}e^{-t\tau_k^2}e_{\tau_k}(x)e_{\tau_k}(y)|dt\\ \nonumber
		&\ls \int_0^{\la^{-2}} t^{-\frac n2}e^{-cd_g(x,y)^2/t}dt\\\nonumber
			&\ls \begin{cases}
				\log(2+(\la d_g(x,y))^{-1})(1+\la d_g(x,y))^{-N},\ n=2\\\nonumber
				d_g(x,y)^{2-n}(1+\la d_g(x,y))^{-N},\quad\quad\quad\quad\quad n\ge3
			\end{cases}\\ \nonumber
			&\ls 
			W_n(d_g(x,y))(1+\la d_g(x,y))^{-N},\ \forall N.\end{align}
		By Corollary \ref{rough}
		\begin{align}\label{specsum}\sum_{\la_j}|h(\la_j)\beta(\la_j\ls\la)||e_j^0(x)e_j^0(y)|\ls \la^n.\end{align}
		Again, as in \eqref{m32}, using \eqref{ker02} and $V\in\mathcal{K}(M)$ we may choose $\Lambda(\eps,V)<\infty$ so that for $\la>\Lambda(\eps,V)$
		\begin{align}\label{spec02}
			|\int_M\sum_{\la_j}\sum_{\tau_k}M_{34}(\la_j)e_j^0(x)e_j^0(y)e_{\tau_k}(x)e_{\tau_k}(y)V(y)dy|\ls\eps\la^n.
		\end{align}
		So the proof is complete.

		\section{Sharp pointwise Weyl Law}
		In this section, we prove Theorem \ref{thm2}: the sharp pointwise Weyl Law for $-\Delta_g+V$
		\[\ola(P_V)(x,x)=\frac{\omega_n}{(2\pi)^{n}}\lambda^{n}+O(\lambda^{n-1}).\]

		From now on, we fix $\eps=\la^{-1}$ in the definitions of $h(\tau)$ and $\tilde\chi_\la(\tau)$ in \eqref{defin}, \eqref{defchi}. They satisfy the following rapid decay properties: for $\tau>0$
		\[|h(\tau)-\ola(\tau)|\ls(1+|\tau-\la|)^{-N},\ \forall N,\] \begin{equation}\label{hdao'}
			|\partial_\tau^j h(\tau)|\ls (1+|\tau-\la|)^{-N},\ \forall N,\ j=1,2,....
		\end{equation}
		\begin{equation}\label{cdao'}|\partial_{\tau}^j\tilde\chi_\la(\tau)|\ls (1+|\tau-\la|)^{-N}, \forall N,\ j=0,1,2,....\end{equation} 
		
		We shall need the following lemma whose proof we postpone to the next section.
		\begin{lemma}\label{hdiff'} Let $n\ge2$ and $V\in L^n(M)$. Then\[\sup_{x\in M}|(h(P_V)-h(P^0))(x,x)|\le C_V \la^{n-1}.\]
		\end{lemma}
	\begin{lemma}[Spectral projection bounds for $H_V$, \cite{BSS}]\label{bssbound}Let $n\ge2$, if $\,V\in \mathcal{K}(M)\cap L^{n/2}(M)$, then for $\la\ge1$, we have
	\begin{equation}
		\sup_{x\in M}\1_{[\la,\la+1)}(P_V)(x,x)\ls \la^{n-1}.
	\end{equation}
\end{lemma}	
The condition  $L^{n/2}$ in Lemma \ref{bssbound} can be dropped when $n=2,3$, see \cite{BSS}, \cite{fs}.

Using the classical pointwise Weyl Law for $P^0$ (see e.g. \cite{SoggeHangzhou})
		\[h(P^0)(x,x)=\frac{\omega_n}{(2\pi)^{n}}\lambda^{n}+O(\lambda^{n-1}).\]
		Then by Lemma \ref{hdiff'}, we get
		\[h(P_V)(x,x)=\frac{\omega_n}{(2\pi)^{n}}\lambda^{n}+O(\lambda^{n-1}).\]
		Since $V\in L^n(M)\subset \mathcal{K}(M)\cap L^{n/2}(M)$, we have \[\1_{[\la,\la+1)}(P_V)(x,x)=O(\la^{n-1}),\]
		which follows from the spectral projection bounds for $H_V$ (Lemma \ref{bssbound}).
		
		Hence
		\begin{align*}\Big|(\ola(P_V)-h(P_V))(x,x)\Big|&\le \sum_{k}|\ola(\tau_k)-h(\tau_k)||e_{\tau_k}(x)|^2\\
			&\ls  \sum_{\mu=0}^\infty (1+|\mu-\la|)^{-N}\sum_{\tau_k\in[\mu,\mu+1)}|e_{\tau_k}(x)|^2\\
			&\ls \sum_{\mu=0}^\infty (1+|\mu-\la|)^{-N}\cdot (1+\mu)^{n-1}\\
			&\ls \lambda^{n-1}.\end{align*}
		So  \[\ola(P_V)(x,x)=\frac{\omega_n}{(2\pi)^n}\la^n+O(\la^{n-1}).\]
		To complete the proof, we only need to prove Lemma \ref{hdiff'}.
		\section{Proof of Lemma  \ref{hdiff'}}
		By the same reduction argument using the Duhamel's principle  as in Section 2, it suffices to show \begin{align}\label{main2}\Bigl| \, \sum_{\lambda_j}
			\sum_{\tau_k}\int_M
			\frac{h(\la_j)-h(\tau_k)}{\la_j^2-\tau_k^2} e_j^0(x)e_j^0(y)e_{\tau_k}(x)
			e_{\tau_k}(y)V(y) \, dy\, \Bigr|\ls \|V\|_{L^n(M)}\la^{n-1}.\end{align}
		Split the sum into two parts\[\sum_{\lambda_j}
		\sum_{\tau_k}=\sum_{\tau_k\le10\la}\sum_{\la_j}+\sum_{\tau_k>10\la}\sum_{\la_j}\]
		We first deal with the high-frequency part. Note that when $\tau_k>10\la$, we may simply fix $\eps=\la^{-1}$  in the proof of the high-frequency part of Lemma \ref{hdiff} (only assuming $V\in\mathcal{K}(M)$) to get the desired bound $O(\la^{n-1})$, except in \eqref{spec01}, \eqref{spec03}, \eqref{spec02}. The reason is that the kernel estimates in these three terms have some negative powers of $\eps$, which are not good enough if $\eps=\la^{-1}$. To get around the difficulty, we need to use the condition $V\in L^n(M)$.
		
		 First, we handle \eqref{spec02} by using H\"older inequality and the estimates \eqref{ker02}  we get
		\begin{align}
			|\int_M\sum_{\la_j}\sum_{\tau_k}&\int_0^{\la^{-2}} h(\la_j)e^{t(\la_j^2-\tau_k^2)}dt\beta(\la_j\ls\la)e_j^0(x)e_j^0(y)e_{\tau_k}(x)e_{\tau_k}(y)V(y)dy|\\
			&\ls \la^n\cdot\|V\|_{L^n(M)}\cdot \|\int_0^{\la^{-2}}\Big|\sum_{\tau_k}e^{-t\tau_k^2}e_{\tau_k}(x)e_{\tau_k}(\cdot)\Big|dt\|_{L^{\frac{n}{n-1}}(M)}\\
			&\ls \|V\|_{L^n(M)} \la^{n-1},
		\end{align}
		where in the last inequality we used \eqref{ker02} and the fact that 
		 \begin{equation}\label{int01}\| \la^{n-2}W_n(\la d_g(x,\cdot))(1+\la d_g(x,\cdot))^{-N} \|_{L^{\frac{n}{n-1}}(M)}\ls \la^{-1}, \,\,\,\text{if}\,\,\, N>n.\end{equation}

		To handle \eqref{spec01} and \eqref{spec03}, we need to use the following lemma, whose proof can be found in the Appendix. Throughout this paper, we use the convention that  $L^{\frac{2n}{n-2}}(M)$ means $L^\infty(M)$ if $n=2$.	
		\begin{lemma}\label{delta2}  Let $I\subset \R_+$ and for
			eigenvalues $\tau_k\in I$ assume that $\delta_{\tau_k}
			\in [0,\delta]$.  Then if $m\in C^1(\R_+\times M)$, and $V\in L^{n}(M)$, we have
			\begin{multline}\label{2.26}
				\int_M \Bigl| \sum_{\tau_k\in I}m(\delta_{\tau_k},y) \,
				a_k V(y)e_{\tau_k}(y)\Bigr| \, dy
				\\
				\le \|V\|_{L^n(M)}\cdot \Bigl(\, \|m(0, \, \cdot\, )\|_{L^{\frac{2n}{n-2}}(M)}
				+\int_0^\delta \bigl\| \tfrac\partial{\partial s}
				m(s, \, \cdot\, )\bigr\|_{L^{\frac{2n}{n-2}}(M)}\, ds
				\, \Bigr) \times  (\sum_{\tau_k\in I}|a_k|^2)^\frac12.
			\end{multline}
		\end{lemma}
	
		Second, decompose $(10\la,\infty)=\bigcup_{\ell\ge0}I_\ell$, where \[I_\ell=(10\cdot2^\ell\la,10\cdot 2^{\ell+1}\la].\] Then by classical Sobolev estimates
		\begin{equation}\label{sobo}
		\begin{aligned}
		\|\sum_{\la_j}\frac{h(\la_j)e^{\la^{-2}\la_j^2}}{\la_j^2-(10\cdot 2^\ell\la)^2}&\beta(\la_j\ls \la)e_j^0(x)e_j^0(\cdot)\|_{L^{\frac{2n}{n-2}}(M)} \\ &
		\ls\la\cdot\|\sum_{\la_j}\frac{h(\la_j)e^{\la^{-2}\la_j^2}}{\la_j^2-(10\cdot 2^\ell\la)^2}\beta(\la_j\ls \la)e_j^0(x)e_j^0(\cdot)\|_{L^2(M)} \\ &
		\ls (2^\ell\la)^{-2}\la^{n/2+1}, 
		\end{aligned}
		\end{equation}
		and similarly, for $s\in I_\ell$
		\[\|\frac{\partial}{\partial s}\sum_{\la_j}\frac{h(\la_j)e^{\la^{-2}\la_j^2}}{\la_j^2-s^2}\beta(\la_j\ls \la)e_j^0(x)e_j^0(\cdot)\|_{L^{\frac{2n}{n-2}}(M)} \ls (2^\ell\la)^{-3}\la^{n/2+1}. \]
		By Lemma~\ref{delta2} with
		$\delta=10\cdot2^\ell\la$, and Corollary \ref{rough} , we have
		\begin{align*}
			|\int_M\sum_{\la_j}&\sum_{\tau_k\in I_\ell}\frac{h(\la_j)e^{\la^{-2}(\la_j^2-\tau_k^2)}}{\la_j^2-\tau_k^2}\beta(\la_j\ls\la)e_j^0(x)e_j^0(y)e_{\tau_k}(x)e_{\tau_k}(y)V(y)dy|\\
			&\ls \|V\|_{L^n(M)}\cdot (2^\ell\la)^{-2}\la^{n/2}\cdot\la\cdot  (\sum_{\tau_k\in I_\ell}|e^{-\la^{-2}\tau_k^2}e_{\tau_k}(x)|^2)^\frac12\\
			&\ls \|V\|_{L^n(M)}\cdot (2^\ell\la)^{-2}\la^{n/2}\cdot\la\cdot e^{-2^{2\ell}}\cdot (2^\ell\la)^{n/2}\\
			&\ls \|V\|_{L^n(M)}\la^{n-1}\cdot e^{-2^{2\ell}}2^{(n/2-2)\ell}.
		\end{align*}
		Summing over $\ell$, we get
		\begin{align*}
			|\int_M\sum_{\la_j}&\sum_{\tau_k> 10\la}\frac{h(\la_j)e^{\la^{-2}(\la_j^2-\tau_k^2)}}{\la_j^2-\tau_k^2}\beta(\la_j\ls\la)e_j^0(x)e_j^0(y)e_{\tau_k}(x)e_{\tau_k}(y)V(y)dy|\ls \|V\|_{L^n(M)}\la^{n-1}.
		\end{align*}
		Third, for $\tau_k\le10\la$, 
		\[\|\sum_{\la_j}\int_0^{\la^{-2}} h(\la_j)e^{t(\la_j^2-\tau_k^2)}dt\beta(\la_j\ls\la)e_j^0(x)e_j^0(\cdot)\|_{L^2(M)}\ls \la^{-2}\la^{n/2}\]
		and for $s\in[1,10\la]$
		\[\|\frac{\partial}{\partial s}\sum_{\la_j}\int_0^{\la^{-2}} h(\la_j)e^{t(\la_j^2-s^2)}dt\beta(\la_j\ls\la)e_j^0(x)e_j^0(\cdot)\|_{L^2(M)}\ls \la^{-3}\la^{n/2}\]
		By Lemma~\ref{delta2} with
		$\delta=10\la-1$, and using Sobolev estimates and Corollary \ref{rough}  as before, we have
		\begin{align*}
			|\int_M\sum_{\la_j}&\sum_{\tau_k\in[1,10\la]}\int_0^{\la^{-2}} h(\la_j)e^{t(\la_j^2-\tau_k^2)}dt\beta(\la_j\ls\la)e_j^0(x)e_j^0(y)e_{\tau_k}(x)e_{\tau_k}(y)V(y)dy|\\
			&\ls \|V\|_{L^n(M)}\cdot \la^{-2}\la^{n/2}\cdot\la\cdot (\sum_{\tau_k\in [1,10\la]}|e_{\tau_k}(x)|^2)^\frac12\\
			&\ls \|V\|_{L^n(M)}\cdot \la^{-2}\la^{n/2}\cdot\la\cdot \la^{n/2}\\
			&\ls \|V\|_{L^n(M)}\la^{n-1}.
		\end{align*}
		So we finish dealing with \eqref{spec01} and \eqref{spec03}.
		As a result,  we obtain
		\[\Bigl| \, \sum_{\lambda_j}
		\sum_{\tau_k>10\la}\int_M
		\frac{h(\la_j)-h(\tau_k)}{\la_j^2-\tau_k^2} e_j^0(x)e_j^0(y)e_{\tau_k}(x)
		e_{\tau_k}(y)V(y) \, dy\, \Bigr|\ls \|V\|_{L^n(M)}\la^{n-1}.\]
		Therefore,  we only need to deal with $\tau_k\le 10\la$. We want to show
		\begin{multline}\label{mid}\Bigl| \, \sum_{\lambda_j}
			\sum_{\la/2\le \tau_k\le10\la}\int_M
			\frac{h(\la_j)-h(\tau_k)}{\la_j^2-\tau_k^2} e_j^0(x)e_j^0(y)e_{\tau_k}(x)
			e_{\tau_k}(y)V(y) \, dy\, \Bigr| \\ \ls \|V\|_{L^n(M)}\la^{n-1},\end{multline}
		as well as 
		\begin{equation}\label{low}\Bigl| \, \sum_{\lambda_j}
			\sum_{\tau_k<\la/2}\int_M
			\frac{h(\la_j)-h(\tau_k)}{\la_j^2-\tau_k^2} e_j^0(x)e_j^0(y)e_{\tau_k}(x)
			e_{\tau_k}(y)V(y) \, dy\, \Bigr|\ls  \|V\|_{L^n(M)}\la^{n-1}.\end{equation}
			
	\noindent \textbf{1. Low-frequency ($\tau_k<\la/2$).}
		First, we prove the low-frequency estimate \eqref{low}. We may choose smooth cut-off functions to decompose
		\[1=\beta(\la_j\ls \la)+\beta(\la_j\approx \la)+\beta(\la_j\gs \la)\]
		where $\beta(\la_j\ls \la)$ is supported on $\{\la_j<3\la/4\}$, and $\beta(\la_j\gs \la)$ is supported on $\{\la_j>2\la\}$.
		By the rapid decay properties of $h(\tau)$ for $\tau>0$
		\[|h(\tau)-\ola(\tau)|\ls(1+|\tau-\la|)^{-N},\ \forall N,\]
		\[|\partial_{\tau}^\ell h(\tau)|\ls (1+|\tau-\la|)^{-N},\ \forall N,\ \ell=1,2,...,\] we have for any $\sigma>0$
		\begin{multline}\nonumber \Bigl| \, \sum_{\lambda_j}
		\sum_{\tau_k<\la/2}\int_M
		\frac{h(\la_j)-h(\tau_k)}{\la_j^2-\tau_k^2}\beta(\la_j\ls\la) e_j^0(x)e_j^0(y)e_{\tau_k}(x)
		e_{\tau_k}(y)V(y) \, dy\, \Bigr| \\ \ls \|V\|_{L^n(M)}\la^{-\sigma},\end{multline}
		\begin{multline}\nonumber \Bigl| \, \sum_{\lambda_j}
		\sum_{\tau_k<\la/2}\int_M
		\frac{1-h(\tau_k)}{\la_j^2-\tau_k^2}(\beta(\la_j\approx\la)) e_j^0(x)e_j^0(y)e_{\tau_k}(x)
		e_{\tau_k}(y)V(y) \, dy\, \Bigr|\\ \ls\|V\|_{L^n(M)}\la^{-\sigma},\end{multline}
		\begin{multline}\nonumber \Bigl| \, \sum_{\lambda_j}
		\sum_{\tau_k<\la/2}\int_M
		\frac{h(\la_j)}{\la_j^2-\tau_k^2}\beta(\la_j\gs\la) e_j^0(x)e_j^0(y)e_{\tau_k}(x)
		e_{\tau_k}(y)V(y) \, dy\, \Bigr| \\ \ls\|V\|_{L^n(M)}\la^{-\sigma}.\end{multline}
		Here we use the mean-value theorem and the rough eigenfunction bounds (Corollary \ref{rough}). So it remains to show
		\begin{multline}\label{lowhigh}\Bigl| \, \sum_{\lambda_j}
			\sum_{\tau_k<\la/2}\int_M
			\frac{h(\tau_k)}{\la_j^2-\tau_k^2}\beta(\la_j\gs\la) e_j^0(x)e_j^0(y)e_{\tau_k}(x)
			e_{\tau_k}(y)V(y) \, dy\, \Bigr| \\
			\ls \|V\|_{L^n(M)}\la^{n-1},\end{multline}
		and 
		\begin{multline}\label{lowmid}
		\Bigl| \, \sum_{\lambda_j}
			\sum_{\tau_k<\la/2}\int_M
			\frac{h(\la_j)-1}{\la_j^2-\tau_k^2}\beta(\la_j\approx\la) e_j^0(x)e_j^0(y)e_{\tau_k}(x)
			e_{\tau_k}(y)V(y) \, dy\, \Bigr| \\
			\ls \|V\|_{L^n(M)}\la^{n-1}.\end{multline}
		To deal with \eqref{lowhigh}, we note that if $\tau_k\le \la/2$, then
		\[\tilde M_4:\mu\mapsto \frac{1}{\mu^2-(\tau_k/\la)^2}\beta(\mu\gs1)\]
		is a symbol in $S^{-2}$ and satisfies
		\[|\partial_{\mu}^\ell \tilde M_4(\mu)|\ls (1+\mu)^{-2-\ell},\ \ell=0,1,2,....\]
		Here $\beta(\mu\gs1)$ is supported on $\{\mu>2\}$. So we get the kernel estimate
		\begin{align} \label{kerlowh}|\sum_{\lambda_j}
			\frac{1}{\la_j^2-\tau_k^2}\beta(\la_j\gs\la) &e_j^0(x)e_j^0(y)|\\  \nonumber
			&=\la^{-2}\tilde M_4(P^0/\la)(x,y)\\ \nonumber
			&\ls \begin{cases}
				\log(2+(\la d_g(x,y))^{-1})(1+\la d_g(x,y))^{-N},\ n=2\\
				d_g(x,y)^{2-n}(1+\la d_g(x,y))^{-N},\quad\quad\quad\quad\quad n\ge3.\end{cases} \nonumber
		\end{align}
		Then by H\"older inequality and Corollary \ref{rough} 
		\begin{align*}
			\Bigl| \, \sum_{\lambda_j}
			&\sum_{\tau_k<\la/2}\int_M
			\frac{h(\tau_k)}{\la_j^2-\tau_k^2}\beta(\la_j\gs\la) e_j^0(x)e_j^0(y)e_{\tau_k}(x)
			e_{\tau_k}(y)V(y) \, dy\, \Bigr|\\
			&\ls \la^n\cdot\|V\|_{L^n(M)}\cdot\|\la^{-2}\tilde M_4(P^0/\la)(x,\cdot)\|_{L^{\frac{n}{n-1}}}\\
			&\ls \|V\|_{L^n(M)}\la^{n-1}.
		\end{align*}Here in the last step we use \eqref{int01} again.
		Moreover, to deal with \eqref{lowmid}, we need to use Lemma \ref{delta2} with $\delta=\la/2-1$.
		Since by standard Sobolev estimates
		\[\|\sum_{\lambda_j}
		\frac{h(\la_j)-1}{\la_j^2-1}\beta(\la_j\approx\la) e_j^0(x)e_j^0(\cdot)\|_{L^{\frac{2n}{n-2}}(M)}\ls \la^{-2}\la^{n/2+1},\]
		and for $s\in[1,\la/2]$
		\[\|\frac{\partial}{\partial s}\sum_{\lambda_j}
		\frac{h(\la_j)-1}{\la_j^2-s^2}\beta(\la_j\approx\la) e_j^0(x)e_j^0(\cdot)\|_{L^{\frac{2n}{n-2}}(M)}\ls \la^{-3}\la^{n/2+1},\]
		by using Corollary \ref{rough}  we get
		\begin{align*}
			\Bigl| \, \sum_{\lambda_j}
			&\sum_{\tau_k<\la/2}\int_M
			\frac{h(\la_j)-1}{\la_j^2-\tau_k^2}\beta(\la_j\approx\la) e_j^0(x)e_j^0(y)e_{\tau_k}(x)
			e_{\tau_k}(y)V(y) \, dy\, \Bigr|\\
			&\ls \|V\|_{L^n(M)}\cdot \la^{-2}\la^{n/2}\cdot\la\cdot (\sum_{\tau_k\in [1,\la/2)}|e_{\tau_k}(x)|^2)^\frac12\\
			&\ls \|V\|_{L^n(M)}\cdot \la^{-2}\la^{n/2}\cdot\la\cdot \la^{n/2}\\
			&\ls \|V\|_{L^n(M)}\la^{n-1}.
		\end{align*}
		So we finish the proof of \eqref{low}.
		
	\noindent \textbf{2. Middle-frequency ($\la/2\le \tau_k\le 10\la$).}
	Next, we prove the middle-frequency estimate \eqref{mid}. We choose smooth cut-off functions to decompose
		\[1=\beta(|\la_j-\tau_k|\ls1)+\sum_{\ell\in \mathbb{N}:2^\ell\le\frac{\la}{100}}\beta(|\la_j-\tau_k|\approx 2^\ell)+\beta(|\la_j-\tau_k|\gs\la).\]
Here $\beta(x\ls1)$ is supported on $\{x< \frac12\}$, and $\beta(x\approx 2^\ell)$ is supported on $\{2^{\ell-2}<x<2^\ell\}$.
		Let 
		\[K_{\tau_k}(x,y)=\sum_{\la_j}\frac{h(\la_j)-h(\tau_k)}{\la_j^2-\tau_k^2}\ e_j^0(x)e_j^0(y)\]
		\[K_{\tau_k,0}(x,y)=\sum_{\la_j}\frac{h(\la_j)-h(\tau_k)}{\la_j^2-\tau_k^2}\beta(|\la_j-\tau_k|\ls1) e_j^0(x)e_j^0(y)\]
		\[K_{\tau_k,\ell}^+(x,y)=\sum_{\la_j}\frac{h(\la_j)}{\la_j^2-\tau_k^2}\beta(|\la_j-\tau_k|\approx 2^\ell) e_j^0(x)e_j^0(y)\]
		\[R_{\tau_k,\ell}(x,y)=\sum_{\la_j}\frac{1}{\la_j^2-\tau_k^2}\beta(|\la_j-\tau_k|\approx 2^\ell) e_j^0(x)e_j^0(y)\]
		\[K_{\tau_k,\infty}^+(x,y)=\sum_{\la_j}\frac{h(\la_j)}{\la_j^2-\tau_k^2}\beta(|\la_j-\tau_k|\gs \la) e_j^0(x)e_j^0(y)\]
		\[R_{\tau_k,\infty}(x,y)=\sum_{\la_j}\frac{1}{\la_j^2-\tau_k^2}\beta(|\la_j-\tau_k|\gs \la) e_j^0(x)e_j^0(y)\]
		\[K_{\tau_k,\ell}^-(x,y)=\sum_{\la_j}\frac{h(\la_j)-1}{\la_j^2-\tau_k^2}\beta(|\la_j-\tau_k|\approx 2^\ell) e_j^0(x)e_j^0(y)\]
		\[K_{\tau_k,\infty}^-(x,y)=\sum_{\la_j}\frac{h(\la_j)-1}{\la_j^2-\tau_k^2}\beta(|\la_j-\tau_k|\gs \la) e_j^0(x)e_j^0(y)\]
		When $\tau_k\in(\la,10\la]$, we decompose
		\[K_{\tau_k}=K_{\tau_k,0}+\sum_{2^\ell\le\la/100}(K_{\tau_k,\ell}^+-R_{\tau_k,\ell}h(\tau_k))+K_{\tau_k,\infty}^+-R_{\tau_k,\infty}h(\tau_k)\]
		and when $\tau_k\in[\la/2,\la]$, we decompose
		\[K_{\tau_k}=K_{\tau_k,0}+\sum_{2^\ell\le\la/100}(K_{\tau_k,\ell}^-+R_{\tau_k,\ell}(1-h(\tau_k)))+K_{\tau_k,\infty}^-+R_{\tau_k,\infty}(1-h(\tau_k))\]
		Moreover, we decompose \[1=\eta(\la_j\gs\tau_k)+\eta(\la_j\ls\tau_k)\] where $\eta(\la_j\gs \tau_k)$ is supported on $\{\la_j\ge 2\tau_k\}$, and then write
		\[K_{\tau_k,\infty}^-=H_{\tau_k,\infty}^-.
		+\tilde K_{\tau_k,\infty}^-\]where
		\[H_{\tau_k,\infty}^-(x,y)=-\sum_{\la_j}\frac{\eta(\la_j\gs \tau_k)}{\la_j^2-\tau_k^2}\beta(|\la_j-\tau_k|\gs \la) e_j^0(x)e_j^0(y)\]
		\begin{multline}\nonumber \tilde K_{\tau_k,\infty}^-(x,y)=\sum_{\la_j}\frac{h(\la_j)}{\la_j^2-\tau_k^2}\beta(|\la_j-\tau_k|\gs \la) e_j^0(x)e_j^0(y) \\-\sum_{\la_j}\frac{\eta(\la_j\ls \tau_k)}{\la_j^2-\tau_k^2}\beta(|\la_j-\tau_k|\gs \la) e_j^0(x)e_j^0(y).\end{multline}
		For $\ell \in {\mathbb N}$ with $2^\ell \le \la/100$, let for $j=0,1,2,\dots$
		\begin{equation}\label{3.5'}
			I^-_{\ell,j}=\bigl(\la-(j+1)2^\ell, \, \la
			-j2^\ell\bigr]
			\quad \text{and} \quad
			I^+_{\ell,j}=\bigl(\la+j2^\ell, \,
			\la+(j+1)2^\ell \, \bigr].
		\end{equation}
		
		By the spectral projection estimates and the rapid decay property of $h(\la_j)$, we have the following lemma.
		\begin{lemma}\label{K3}
			If $\ell \in {\mathbb Z}_+$, $2^\ell\le \la/100$,
			and $j=0,1,2,\dots$, we have for each $N\in {\mathbb N}$
			\begin{multline}\label{3.6}
				\|K^{\pm}_{\tau,\ell}(x, \, \cdot )\|_{L^{\frac{2n}{n-2}}(M)}, \, \,
				\|2^\ell \tfrac\partial{\partial\tau}K^{\pm}_{\tau,\ell}(x, \, \cdot  )\|_{L^{\frac{2n}{n-2}}(M)}
				\\
				\lesssim \la^{\frac{n}2-1}(1+j)^{-N},
				\quad \tau\in I_{\ell,j}^\pm \cap [\la/2,10\la].
			\end{multline}
			Also,
			\begin{multline}\label{3.7}
				\|K_{\tau,0}(x, \, \cdot  )\|_{L^{\frac{2n}{n-2}}(M)}, \, \,
				\|\tfrac\partial{\partial\tau}K_{\tau,0}(x, \, \cdot  )\|_{L^{\frac{2n}{n-2}}(M)}
				\\
				\lesssim
				\la^{\frac{n}2-1}(1+j)^{-N}, \quad
				\tau\in I_{0,j}^\pm \cap [\la/2,10\la],
			\end{multline}
			Moreover,
			we also have for $1\le 2^\ell \le \la/100$ and $\tau\in [\la/2,10\la]$
			\begin{equation}\label{3.10}
				\|R_{\tau,\ell}(x, \, \cdot  )\|_{L^{\frac{2n}{n-2}}(M)}, \, \,
				\|2^\ell \tfrac\partial{\partial \tau}R_{\tau,\ell}(x, \, \cdot )\|_{L^{\frac{2n}{n-2}}(M)}
				\lesssim \la^{\frac{n}2-1},
			\end{equation}
			\begin{equation}\label{3.14}
			|R_{\tau,\infty}(x, y)|\lesssim 
			\begin{cases}
				\log(2+(\la d_g(x,y))^{-1})(1+\la d_g(x,y))^{-N},\ n=2\\
				d_g(x,y)^{2-n}(1+\la d_g(x,y))^{-N},\quad\quad\quad\quad\quad n\ge3.\end{cases}
		\end{equation}
			Finally, we have for $\tau\in(\la,10\la]$
			\begin{equation}\label{3.11}
				\|K_{\tau,\infty}^+(x, \, \cdot  )\|_{L^{\frac{2n}{n-2}}(M)}, \, \,
				\|\la \tfrac\partial{\partial \tau}K_{\tau,\infty}^+(x, \, \cdot )\|_{L^{\frac{2n}{n-2}}(M)}
				\lesssim \la^{\frac{n}2-1}
			\end{equation}
			and for $\tau\in[\la/2,\la]$
			\begin{equation}\label{3.12}
				\|\tilde K_{\tau,\infty}^-(x, \, \cdot  )\|_{L^{\frac{2n}{n-2}}(M)}, \, \,
				\|\la \tfrac\partial{\partial \tau}\tilde K_{\tau,\infty}^-(x, \, \cdot )\|_{L^{\frac{2n}{n-2}}(M)}
				\lesssim \la^{\frac{n}2-1},
			\end{equation}
			\begin{align}\label{3.13}
				|H_{\tau,\infty}^-(x,y)| &\ls \begin{cases}
					\log(2+(\la d_g(x,y))^{-1})(1+\la d_g(x,y))^{-N},\ n=2\\
					d_g(x,y)^{2-n}(1+\la d_g(x,y))^{-N},\quad\quad\quad\quad\quad n\ge3.
				\end{cases}
			\end{align}
		
		\end{lemma}
We postpone the proof of Lemma~\ref{K3} to the Appendix.

	By Lemma~\ref{delta2} with
		$\delta=2^\ell$ along with the Lemma~\ref{K3} and Lemma~\ref{bssbound}, we have
		\begin{align}\label{3.15}
			\Bigl| &\sum_{\tau_k\in I^\pm_{\ell,j}\cap[\la/2,10\la]}
			\int K^\pm_{\tau_k,\ell}(x,y)e_{\tau_k}(x)e_{\tau_k}(y)
			V(y)\, dy\Bigr|
			\\
			&\lesssim \|V\|_{L^{n}(M)}\cdot \sup_x
			\Bigl( \|K^\pm_{\la\pm j2^\ell,\ell}(x, \, \cdot )\|_{L^{\frac{2n}{n-2}}(M)}+\int_{I^\pm_{\ell,j}}\bigl\|\tfrac\partial{\partial s}K^\pm_{s,\ell}(x, \, \cdot)\|_{L^{\frac{2n}{n-2}}(M)} \, ds
			\Bigr) \notag
			\\
			&\qquad \qquad \qquad \qquad \qquad \times  \bigl(\sum_{\tau_k\in I^\pm_{\ell,j}\cap [\la/2,10\la]}
			|e_{\tau_k}(x)|^2\bigr)^{1/2}
			\notag
			\\
			&\lesssim \la^{\frac{n}2-1}(1+j)^{-N}
			\cdot \la^\frac{n-1}22^{\ell/2}\cdot \|V\|_{L^{n}(M)}
			\notag
			\\
			&\lesssim \la^{n-\frac32}2^{\ell/2}(1+j)^{-N} \cdot \|V\|_{L^{n}(M)}.  \notag
		\end{align}
		If we sum over $j=0,1,2,\dots$, we see that \eqref{3.15} yields
		\begin{multline}\label{3.17}
			\Big| \sum_{\la <\tau_k\le 10\la}
			\int K^+_{\tau_k,\ell}(x,y) e_{\tau_k}(x)e_{\tau_k}(y)
			\, V(y)\, dy\Bigr|
			\\
			+\Big| \sum_{\la/2 \le \tau_k\le\la}
			\int K^-_{\tau_k,\ell}(x,y) e_{\tau_k}(x)e_{\tau_k}(y)
			\, V(y)\, dy\Bigr|
			\lesssim \|V\|_{L^{n}(M)} \la^{n-\frac32}2^{\ell/2}.
		\end{multline}
		If we take $\delta=1$ in Lemma \ref{delta2}, this argument also gives
		\begin{multline}\label{3.19}
			\Big| \sum_{\la <\tau_k\le 10\la}
			\int K_{\tau_k, 0}(x,y) e_{\tau_k}(x)e_{\tau_k}(y)
			\, V(y)\, dy\Bigr|
			\\
			+\Big| \sum_{\la/2 \le \tau_k\le\la}
			\int K_{\tau_k, 0}(x,y) e_{\tau_k}(x)e_{\tau_k}(y)
			\, V(y)\, dy\Bigr|
			\lesssim \|V\|_{L^{n}(M)} \la^{n-\frac32}.
		\end{multline}
		If we take $\delta=\la$ in Lemma \ref{delta2}, this argument also gives
		\begin{align}\label{3.20}
			\Big|\sum_{\la <\tau_k\le 10\la}
			\int K^+_{\tau_k,\infty}(x,y) e_{\tau_k}(x)e_{\tau_k}(y)V(y)dy\Big|\ls \|V\|_{L^{n}(M)} \la^{n-1}.
		\end{align}
		Similarly,
		\begin{align}\label{3.20'}
			\Big| \sum_{\la/2 <\tau_k\le \la}
			\int \tilde K^-_{\tau_k,\infty}(x,y) e_{\tau_k}(x)e_{\tau_k}(y)
			\, V(y)\, dy\Bigr|	\lesssim \|V\|_{L^{n}(M)} \la^{n-1}.
		\end{align}
		By \eqref{3.10}, if we repeat the argument above, we have
		\begin{multline}\label{3.21}
			\Big| \sum_{\la <\tau_k\le 10\la}
			\int R_{\tau_k, \ell}(x,y)  h(\tau_k)e_{\tau_k}(x)e_{\tau_k}(y)
			\, V(y)\, dy\Bigr|
			\\
			+\Big| \sum_{\la/2 \le \tau_k\le\la}
			\int R_{\tau_k, \ell}(x,y)(1-h(\tau_k)) e_{\tau_k}(x)e_{\tau_k}(y)
			\, V(y)\, dy\Bigr|
			\lesssim \|V\|_{L^{n}(M)} \la^{n-\frac32}2^{\ell/2}.
		\end{multline}
		Moreover, by using \eqref{3.13}, we have $\|H_{\tau,\infty}^-(x,\cdot)\|_{L^{\frac{n}{n-1}}(M)}\ls \la^{-1}$ for $\tau\in[\la/2,10\la]$, and then
		\begin{align*}
			\Big| \sum_{\la/2 \le \tau_k\le\la}\int H_{\tau_k, \infty}^-(x,y)&e_{\tau_k}(x)e_{\tau_k}(y)V (y)dy\Bigr| \\&\ls \|V\|_{L^n(M)}\cdot \la^{-1}\|\sum_{\tau_k\in[\la/2,10\la]}|e_{\tau_k}(x)e_{\tau_k}(\cdot)|\|_\infty
			\\ &\ls \|V\|_{L^n(M)}\la^{n-1}.
		\end{align*}
		Similarly,
		\begin{multline}\label{3.22}
			\Big| \sum_{\la <\tau_k\le 10\la}
			\int R_{\tau_k, \infty}(x,y)  h(\tau_k)e_{\tau_k}(x)e_{\tau_k}(y)
			\, V(y)\, dy\Bigr|
			\\
			+\Big| \sum_{\la/2 \le \tau_k\le\la}
			\int R_{\tau_k, \infty}(x,y)(1-h(\tau_k)) e_{\tau_k}(x)e_{\tau_k}(y)
			\,V (y)\, dy\Bigr|
			\lesssim \|V\|_{L^n(M)}\la^{n-1}.
		\end{multline}
		Hence, using the estimates above and summing over $\ell$, we obtain \eqref{mid}.
	\section{Appendix: Proof of Lemmas}
	We give the proof of Lemma \ref{K3} and Lemma \ref{delta2}. They are essentially analogous to the lemmas in \cite{hs}, but we prove them here for the sake of completeness.
		\begin{proof}[Proof of Lemma~\ref{K3}]
		To prove Lemma~\ref{K3}, we shall need the fact that, by Lemma~\ref{sogge888}, for any fixed $\ell$ with $1\le2^\ell\le \la/100$, we have the following spectral projection estimates (Lemma \ref{sogge888})
		\begin{equation}\label{sper} \|\1_{[\la,\la+2^\ell)}(P^0)\|_{L^2\to L^{\frac{2n}{n-2}}}\ls 2^{\ell/2}\la^{1/2}.
\end{equation}

		To prove the first inequality we note that if $\tau\in I^\pm_{\ell,j}\cap[1,10\la]$ then
		$|\la_i-\tau|\le 2^{\ell+1}$ if $\beta(2^{-\ell}(\la_i-\tau))\ne 0$, and, in this case,
		we also have $h(\la_i)-1=O((1+j)^{-N})$ if $\tau\in I^-_{\ell,j}$ and $h(\la_i)=O((1+j)^{-N})$ if
		$\tau\in I^+_{\ell,j}$.  Therefore, we have
		\begin{align*}
			\bigl\|K^\pm_{\tau,\ell}(\, \cdot \, ,y)\bigr\|_{L^{\frac{2n}{n-2}}(M)} &\ls  2^{\ell/2}\la^{1/2}\bigl\|K^\pm_{\tau,\ell}(\, \cdot \, ,y)\bigr\|_{L^2(M)}\\&\lesssim
			 2^{\ell/2}\la^{1/2}(1+j)^{-N}2^{-\ell}\la^{-1} \, \bigl(\sum_{\{i: \, |\la_i-\tau|\le 2^{\ell+1}\}} |e^0_i(y)|^2\bigr)^{1/2}
			\\
			&\lesssim (1+j)^{-N}2^{-\ell/2}\la^{-1/2}\bigl(\sum_{\mu\in {\mathbb N}: |\mu-\tau|\le 2^{\ell+1}}
			\mu^{n-1}\bigr)^{1/2}
			\\
			&\le (1+j)^{-N}\la^{\frac{n}2}\la^{-1}.
		\end{align*}
		In particular, if $n=2$, the same argument implies that 
		\begin{align*}
			\bigl\|K^\pm_{\tau,\ell}(\, \cdot \, ,y)\bigr\|_{L^{\infty}(M)} \le (1+j)^{-N},
		\end{align*}
		which proves the first part of \eqref{3.6}.
		 The other inequality
		in \eqref{3.6} follows from this argument since
		$$\frac\partial{\partial \tau}
		\frac{\beta(|\la_i-\tau|\approx 2^{\ell} )}{\la_i^2-\tau^2}=O(2^{-2\ell}\la^{-1}),
		$$
		due to the fact that we are assuming that $2^\ell \le \la/100$.
		
		This argument also gives us \eqref{3.7} if we use the fact that
		$\tau\to (h(\tau)-h(\mu))/(\tau^2-\mu^2)$ is smooth and use the fact that
		$$\partial_\tau^k \bigl(\beta(|\la_i-\tau|\ls 1) (h(\la_i)-h(\tau))/(\la_i-\tau)\bigr)=O((1+j)^{-N}), \, \, k=0,1,
		\, \, \tau\in I_{0,j}^\pm.$$
		
		To prove \eqref{3.11} we use the fact that for $k=0,1$ we have for $\tau\in (\la/2,10\la]$
		$$\Bigl| \, \Bigl(\frac\partial{\partial \tau}\Bigr)^k
		\Bigl(\frac{\beta(|\la_i-\tau|\gs \la)}{\la_i^2-\tau^2}
		\Bigr)  h(\la_i)\, \Bigr|
		\lesssim
		\begin{cases}
			\la^{-2-k} \quad \text{if } \, \, \la_i\le \la
			\\
			\la^{-2-k}(1+|\la_i-\la|)^{-N} \quad \text{if } \, \, \la_i>\la.
		\end{cases}
		$$
		Thus for $k=0,1$, by \eqref{sper}
		\begin{align*}
			\bigl\| (\la \partial_\tau)^k \, K^+_{\tau,\infty}(\, \cdot \, ,y)\bigr\|_{L^{\frac{2n}{n-2}}(M)}
			&\lesssim \la\cdot(\sum_{\la_j\le2\la}\la^{-4}|e_j^0(y)|^2)^\frac12+\sum_{s\in \mathbb{N}:2^s>2\la}2^s(\sum_{\la_j\approx 2^s}2^{-4s}(1+|\la_j-\la|)^{-N}|e_j^0(y)|^2)^\frac12
			\\
			&\lesssim \la^{-1+\frac{n}2},
		\end{align*}
		as desired if $N>2n$.
		Similarly, 
		$\tilde K^-_{\tau,\infty}$ satisfies the bounds in \eqref{3.12}.
		
		Moreover, we can conclude from Lemma~\ref{pdo} that $H^-_{\tau,\infty}$ satisfies the bounds
		in \eqref{3.13}.  It just remains to prove the bounds in \eqref{3.10} for the $R_{\tau,\ell}(x,y)$ and that in
		\eqref{3.14} for $R_{\tau,\infty}(x,y)$.  The former just follows from the proof of \eqref{3.6}.

		To prove the remaining inequality, \eqref{3.14}, we note that 
		$$R_{\tau,\infty}(x,y)=R^0_{\tau,\infty}(x,y)-H_{\tau,,\infty}^-(x,y),$$
		if
		$$R^0_{\tau,\infty}(x,y)=\sum_i \eta(\la_i\ls \tau)
		\frac{\beta(|\la_i-\tau|\gs \la)}{\la_i^2-\tau^2}
		\, e^0_i(x)e^0_i(y).$$
		Since $H_{\tau,,\infty}^-$ satisfies \eqref{3.13}, and Lemma~\ref{pdo} shows that 
		$$|R^0_{\tau,\infty}(x,y)|\lesssim \tau^{n-2} \bigl(1+\tau d_g(x,y)\bigr)^{-N}$$
		we conclude that \eqref{3.14} is valid.
	\end{proof}
	
\begin{proof}[Proof of Lemma \ref{delta2}]
	We shall use the fact that
	$$m(\delta_{\tau_k},y)=m(0,y)+\int_0^{\delta}
	\1_{[0,\delta_{\tau_k}]}(s)\, \tfrac\partial{\partial s}
	m(s,y) \, ds,$$
	where $\1_{[0,\delta_{\tau_k}]}(s)$ is the indicator function of the the interval $[0,\delta_{\tau_k}]
	\subset [0,\delta]$.
	Therefore, by H\"older's inequality and Minkowski's inequality,
	the left side of \eqref{2.26} is dominated by $\|V\|_{L^n(M)}$ times
	\begin{align*}
		&\big(\int_M \bigl|\, m(0,y)\cdot \sum_{\tau_k\in I}a_k
		e_{\tau_k}(y)\, \bigr|^{\frac{n}{n-1}} \, dy\big)^{\frac{n-1}{n}} \\
		+& \big(\int_M \bigl| \, \sum_{\tau_k\in I}\int_0^\delta
		\1_{[0,\delta_{\tau_k}]}(s)\tfrac\partial{\partial s}
		m(s,y) a_k e_{\tau_k}(y)\, ds \, \bigr|^{\frac{n}{n-1}} \, dy \big)^{\frac{n-1}{n}}
		\\
		&\le
		\|m(0,\, \cdot\, )\|_{\frac{2n}{n-2}} \cdot
		\|\sum_{\tau_k\in I}a_ke_{\tau_k}\|_{2}
		+\int_0^\delta\bigl( \,
		\bigl\|\tfrac\partial{\partial s}
		m(s,\, \cdot \, )\bigr\|_{\frac{2n}{n-2}} \cdot \bigl\| \sum_{\tau_k\in I}\1_{[0,\delta_{\tau_k}]}(s)a_ke_{\tau_k}\, \bigr\|_{2}\, \bigr)\, ds \\
		&\le
		\|m(0,\, \cdot\, )\|_{\frac{2n}{n-2}} \cdot
		\|\sum_{\tau_k\in I}a_ke_{\tau_k}\|_{2}
		+\int_0^\delta \,
		\|\tfrac\partial{\partial s}
		m(s,\, \cdot \, )\|_{\frac{2n}{n-2}} ds \cdot \sup_{s\in [0, \delta]}\bigl\| \sum_{\tau_k\in I}\1_{[0,\delta_{\tau_k}]}(s)a_ke_{\tau_k}\, \bigr\|_{2}\,\,
		\\
		&\le \Bigl(\, \|m(0, \, \cdot\, )\|_{\frac{2n}{n-2}}
		+\int_0^\delta \bigl\| \tfrac\partial{\partial s}
		m(s, \, \cdot\, )\bigr\|_{\frac{2n}{n-2}}\, ds
		\, \Bigr) \times   (\sum_{\tau_k\in I}|a_k|^2)^\frac12
	\end{align*}
	as desired.
\end{proof}
	 

		\bibliographystyle{plain}

\begin{thebibliography}{0}
			\bibitem{anps} Arendt, W., Nittka, R., Peter, W., and Steiner, F. Weyl’s Law: Spectral Properties of the Laplacian
			in Mathematics and Physics. Mathematical analysis of evolution, information, and complexity (2009): 1-71.
			\bibitem{Avakumovic}V. G. Avakumovi\'c. \"Uber die Eigenfunktionen auf geschlossenen Riemannschen Mannigfaltigkeiten.
			Math. Z., 65:327–344, 1956.
			\bibitem{berard}P. H. B\'erard, On the wave equation on a compact Riemannian manifold without conjugate points,
				Math. Z., 155 (1977), pp. 249-276.
			\bibitem{bhss} M. Blair, X. Huang, Y. Sire, and C. D. Sogge. Uniform Sobolev Estimates on compact manifolds involving singular potentials, Revista Matem\'atica Iberoamericana (2021).
			\bibitem{BSS} M. Blair, Y. Sire, and C. D. Sogge. Quasimode, eigenfunction and spectral projection bounds for
			Schr\"odinger operators on manifolds with critically singular potentials.  Journal of Geometric Analysis 31.7 (2021): 6624-6661.
			\bibitem{bbcs}C. Boccato, C. Brennecke, S. Cenatiempo, and B. Schlein. Bogoliubov theory in the Gross–Pitaevskii limit. Acta Mathematica 222, no. 2 (2019): 219-335.
			\bibitem{carl}T. Carleman. Propri\'et\'es asymptotiques des fonctions fondamentales des membranes vibrantes. Comptes Rendus des Math\'ematiciens Scandinaves \'a Stockholm (1934): 14-18.
			\bibitem{dg}J. J. Duistermaat and V. W. Guillemin, The spectrum of positive elliptic operators and periodic
			bicharacteristics, Invent. Math., 29 (1975), pp. 39-79.
			\bibitem{fs} R.L. Frank and J. Sabin. Sharp Weyl laws with singular potentials. Preprint.
			\bibitem{hla}E. Hlawka,\"Uber Integrale auf konvexen K\"orpern. I, Monatsh. Math., 54 (1950), pp. 1-36.
			\bibitem{hbook}L. H\"ormander. The analysis of linear partial differential operators III. Pseudodifferential operators, Springer-Verlag, Berlin, 1985.
			\bibitem{HSpec}L. H\"ormander. The spectral function of an elliptic operator. Acta Math., 121:193–218, 1968.
			
			\bibitem{hsz}X. Huang, Y. Sire, and C. Zhang. Spectral cluster estimates for Schrödinger operators of relativistic type. Journal de Math\'ematiques Pures et Appliqu\'ees 155 (2021): 32-61.
			\bibitem{hs}X. Huang and C. D. Sogge. Weyl formulae for Schr\" odinger operators with critically singular potentials. Communications in Partial Differential Equations (2021): 1-46.
			\bibitem{hs2}X. Huang and C. D. Sogge. Quasimode and Strichartz estimates for time-dependent Schrödinger equations with singular potentials. arXiv:2011.04007
			\bibitem{hs3} X. Huang and C. D. Sogge. Uniform Sobolev estimates in $\mathbb{R}^n$ involving singular potentials. arXiv:2101.09826
		\bibitem{hz2021} X. Huang and C. Zhang, Sharp pointwise Weyl laws for Schrodinger operators with singular potentials on flat tori, to appear in Commun. Math. Physics.
		
		
			\bibitem{Levitan}B. M. Levitan. On the asymptotic behavior of the spectral function of a self-adjoint differential
			equation of the second order. Izvestiya Akad. Nauk SSSR. Ser. Mat., 16:325–352, 1952.
				\bibitem{Levitan2}B. M. Levitan. On the asymptotic behavior of a spectral function and on expansions in eigenfunctions of a self-adjoint differential
			equation of the second order II. Izvestiya Akad. Nauk SSSR. Ser. Mat., 19:33–58, 1955.
			\bibitem{Liyau}P. Li and S.-T. Yau. On the parabolic kernel of the Schr\"odinger operator. Acta Math., 156(3-4):153–201, 1986.
			
			\bibitem{seeley1}R. Seeley. A sharp asymptotic estimate for the eigenvalues of the
			Laplacian in a domain of R3. Advances in Math., 102(3):244-264	(1978).
			\bibitem{seeley2}R. Seeley. An estimate near the boundary for the spectral function of the Laplace operator. American Journal of Mathematics 102.5 (1980): 869-902.
			\bibitem{SimonSurvey}B. Simon. Schr\"odinger semigroups. Bull. Amer. Math. Soc. (N.S.), 7(3):447–526, 1982.
			\bibitem{sob}A. V. Sobolev, Discrete spectrum asymptotics for the Schr\"odinger operator with a singular potential and a magnetic field, Rev. Math. Phys., 8 (1996), pp. 861-903.
			\bibitem{sogge88}C. D. Sogge. Concerning the Lp norm of spectral clusters for second-order elliptic operators
			on compact manifolds. J. Funct. Anal., 77(1):123-138, 1988.
			\bibitem{fio}C. D. Sogge. Fourier integrals in classical analysis, volume 210 of Cambridge Tracts in Mathematics.
			Cambridge University Press, Cambridge, second edition, 2017.
			\bibitem{SoggeHangzhou} C. D. Sogge. Hangzhou lectures on eigenfunctions of the Laplacian, volume 188 of Annals of Mathematics Studies. Princeton University Press, Princeton, NJ, 2014.
			\bibitem{steinbook} E.M. Stein, Harmonic Analysis: real-variable methods, orthogonality, and osciollatory integrals, Princeton University Press, Princeton, NJ, 1993. 
			\bibitem{sturm}K.-T. Sturm. Schr\"odinger semigroups on manifolds. J. Funct. Anal., 118(2):309–350, 1993.
			\bibitem{weyl}H. Weyl.  \"Uber die asymptotische Verteilung der Eigenwerte. Nachrichten von der Gesellschaft der Wissenschaften zu G\"ottingen, Mathematisch-Physikalische Klasse 1911 (1911): 110-117.
			
		\end{thebibliography}
		
	\end{document}